\documentclass{article}

\usepackage{amsmath,amsxtra,amscd,amssymb,latexsym,stmaryrd,amsthm}
\usepackage{enumitem}
\usepackage{graphicx,color}
\usepackage{hyperref}
\usepackage{appendix}
\usepackage[utf8]{inputenc}
\usepackage[T1]{fontenc}
\usepackage{mathrsfs}
\usepackage[all]{xy}
\usepackage[hang,flushmargin]{footmisc}

\def\supp{\mathop{\rm supp}\nolimits}
\newtheorem{theorem}{Theorem}[section]

\newtheorem{lemma}[theorem]{Lemma}
\newtheorem{proposition}[theorem]{Proposition}
\newtheorem{corollary}[theorem]{Corollary}
\newtheorem{definition}[theorem]{Definition}
\newtheorem{remark}[theorem]{Remark}
\newtheorem{example}[theorem]{Example}

\newtheorem{assumption}[theorem]{Assumption}

\definecolor{Red}{cmyk}{0,1,1,0}

\definecolor{Blue}{cmyk}{1,1,0,0}

\title{Applications  of Variable Discounting Dynamic Programming to
Iterated Function Systems\\ and Related Problems\vspace*{0.7cm}}

\author{\large
  L. Cioletti
  \\[-0.15cm]
  \footnotesize Departamento de Matem\'atica - UnB
  \\[-0.15cm]
  \footnotesize 70910-900, Bras\'ilia, Brazil
  \\[-0.15cm]
  \footnotesize\texttt{cioletti@mat.unb.br}
  \and
  \large
  Elismar R. Oliveira
  \\[-0.15cm]
  \footnotesize Departamento de Matem\'atica - UFRGS
  \\[-0.15cm]
  \footnotesize 91509-900, Porto Alegre, Brazil
  \\[-0.15cm]
  \footnotesize\texttt{elismar.oliveira@ufrgs.br}
}

\date{}

\hypersetup
{
    pdfauthor={L. Cioletti and Elismar R. Oliveira},
    pdfsubject={Dynamic Programming and IFS},
    pdftitle={Applications  of variable discounting dynamic programming to Iterated Function Systems and related problems},
    pdfkeywords={Dynamic programming, IFS, Bellman's equation, Ergodic Optimization}
}

\begin{document}

\makeatletter
\def\blfootnote{\gdef\@thefnmark{}\@footnotetext}
\let\@fnsymbol\@roman
\makeatother

\maketitle

\begin{abstract}
We study existence and uniqueness of the fixed points solutions of
a large class of non-linear variable discounted transfer operators
associated to a sequential decision-making process.
We establish regularity properties of these solutions,
with respect to the immediate return and the variable discount.
In addition, we apply our methods to reformulating and solving,
in the setting of dynamic programming, some central variational problems
on the theory of iterated function systems, Markov decision processes, discrete Aubry-Mather theory, Sinai-Ruelle-Bowen measures, fat solenoidal attractors, and ergodic optimization.
\end{abstract}

\blfootnote{\textup{2010} \textit{Mathematics Subject Classification}: 	37Axx; 37Dxx; 49Lxx}

\section{Introduction}

The abstract theory of dynamic programming (DP for short) is a powerful
tool for analysis of decision-making problems.
This paper aims to strength some of fundamental theorems in this theory
in order to prove new results on existence and uniqueness of variational problems,
arising in Ergodic Theory and Iterated Function Systems,
within a unified framework.

\break

As motivation and to illustrate the applicability of our theorems,
we explain below how to reformulate some of very important
variational problems on:
\begin{itemize}
	\item decision problems for iterated function systems (IFS);
	\item Markov decision process;
	\item discrete Aubry-Mather theory;
	\item Sinai-Ruelle-Bowen (SRB) measures and fat solenoidal attractors;
	\item ergodic optimization;
\end{itemize}
in the language of DP so that their solutions can be obtained by
straightforward applications of our main results.

\bigskip

Since the pioneering work of Bellman \cite{MR0090477}
the application of this theory has been growing fast,
and nowadays it is a well developed subject and a
standard tool for some researchers in pure and applied Mathematics.
It has also been used in engineering problems,
optimal control theory and machine learning, just
to name a few, see
\cite{
MR934442,MR3204932,MR730512,MR1126131,
MR905046,IEEE4358769,MR3005885,MR3248091,
MR1859323,MR2746549}
and references therein.

Successful applications of this theory in Dynamical Systems
were obtained by the so-called {\it discounted methods}.
In \cite{MR1841880} this method is applied to several problems on
thermodynamic formalism as well as in the study of maximizing measures
(ergodic optimization) for expanding endomorphisms on metric spaces.
In \cite{MR2864625,MR3377291} this method was adapted to study
Statistical Mechanics models in one-dimensional one-sided lattices.
In these works, the authors proved existence of particular discounted
limits of solutions of the Bellman equation and obtain
the maximal eigenvalue and a positive eigenfunctions of the Ruelle operator,
and subactions.

Infinite dimensional linear DP problems
are considered in \cite{MR2128794}, \cite{MR2458239}, \cite{MR2761072}
and \cite{MR3135643} in both discrete and continuous setting. In these works a
connection between DP and the theory of viscosity solutions
\cite{MR1473840,MR1451248,MR1650195,MR1650261}
are explored to obtain new results on Aubry-Mather problem \cite{MR1166538,MR1384478} and
related topics as the Hamilton-Jacobi equation.

Here we extended the recently developed
theory of variable discount in DP \cite{MR3248091}
to broaden its range of applications.
Two central problems in our paper are the following ones.
Given a sequential decision-making process $S=\{X, A, \Psi, f, u, \delta\}$
(Definition \ref{def-DMP}), we study the existence and uniqueness of
the fixed point solutions of the variable discounted Bellman's equation
\[
v(x)
=
\sup_{a \in \Psi(x)} u(x, a)+\delta(v(f(x,a)))
\]
as well as the fixed point solutions of the variable discounted transfer operator
\[
w(x)= \ln \int_{a \in \Psi(x)}
\exp\big(\, u(x, a)+\delta(w(f(x,a)))\, \big)\,   d\nu_{x}(a).
\]
In addition, regularity properties of the solutions $v$ and $w$, with respect to $u$
and the variable discount $\delta$ are determined.

After discussing some results on variable discount, we present
new results on aggregator function associated
to the Ruelle operator.
We believe that our results about discounted limits,
in Section \ref{section-discounted-limits},
provide truly new insights into the behavior of decision-making problems.
These insights are clear when the variable discount function vanishes,
because it allows the \textit{future rewards} function play a major role.
It is remarkable fact that the Bellman equation survives
on this general setting and produces
a new equation capable of explaining the
behavior of decision-making problems.

In our opinion this paper will be of potential interest to the readers
working on variational problems in Dynamical Systems such
as ergodic optimization, thermodynamical formalism,
Aubry-Mather theory, Lagrangian mechanics,
Hamilton-Jacobi equations via viscosity solutions, etc.
Nonetheless, some methods presented here
can be useful in Analysis, random dynamics and
many other related fields.

\bigskip

In what follows, we explain within DP framework the statement
of some central problems on the topics mentioned in the beginning
of this section.
Before proceed, we shall introduce
some basic notations. Here $X=(X,d)$ always denotes a complete metric space and
$C(X,\mathbb{R})$, $C_{b}(X,\mathbb{R})$ stands for
the space of all real continuous and real bounded continuous functions
on $X$, respectively. Of course, if $X$ is compact then
$C(X,\mathbb{R})= C_{b}(X,\mathbb{R})$.
Both $C(X,\mathbb{R})$ and $C_{b}(X,\mathbb{R})$
are endowed with their standard supremum norm and regarded as Banach spaces.
The space of all Borel probability measures over $X$ is denoted
by $\mathscr{P}(X)$. If $X$ is a compact space and $T:X\to X$
is a continuous mapping, then we denote by $\mathscr{P}_{T}(X)$ the space of all
$T$-invariant Borel probability measures defined over $X$.
These spaces are endowed
with their standard weak-$*$ topology.

\subsubsection*{Decision Problems for IFS}
A deterministic decision problem controlled by an IFS
$\{\phi_a:X\to X, \, a\in A\}$ can be described as follows.
The state of the system at time $n$ is a point $x_n\in X$
and determined by the following rules.
We give an initial state $x_0$. At each discrete time $n\geq 0$,
a point $a\in A$ (the set of possible actions)
is chosen (by some agent) and the state changes from $x_{n}$
to a new state $x_{n+1}:=\phi_{a}x_{n}$.
There is a reward, given by a real valued function $c(x_{n},a)$,
associated to taking action $a$, when system is in the
state $x_{n}$ and also a discount factor $0<\lambda<1$, which represents
the relevance of the first choices.
In this setting, an infinite horizon decision problem takes the form
\[
V(x_0) \;
=
\displaystyle \;
\sup\left\{
\sum_{n = 0}^{\infty} \lambda^n c(x_n, a_{n}):
\ (a_0,a_1,\ldots)\in A^{\mathbb{N}} \ \text{and }\  x_{n+1}=\phi_{a_n}x_{n}
\right\}
\]
The dynamic programming theory explains how to break this decision problem into smaller subproblems,
leading to Bellman's principle of optimality
\[
V(x)
=
\max_{ a \in A } \left\{ c(x ,a ) + \lambda  V(\phi_{a } x )\right\},
\]
known as Bellman's equation.

\subsubsection*{Markov Decision Process (MDP)}

An example of stochastic dynamic programming problem
is a Markov Decision Process (MDP) controlled by an IFS
$\{\phi_a:X \to X, \, a\in A\}$.
A sample of this decision process is a
feasible history
$(x_0, a_0, x_{1},a_1,\ldots)$, where $x_{n+1}=\phi_{a_{n}} x_{n}$.
In this setting, we fix an ordered quadruple
$(X, A,p,r)$, where $X$ is a set of states,
$A$ is a set of available actions, $p$ is a probability measure such that
$p(x_{n+1}=\phi_{a_{n}} x_{n} | x_{n}=x,a_{n}=a)$ is the probability
that $x$ evolves from $x_{n}=x$  to  the state
$x_{n+1}=\phi_{a} x$, by taking the action $a_{n}=a$,
and $r$ is a function $r: X\times A \to \mathbb{R}$, where $r(x,a)$
is a reward for taking action $a$ at the state $x$.
Note that in this context, the aggregation function will be a random variable.

A central problem in MDP is to find a \textit{policy} for the decision maker,
which is a function $\pi: X \to A^{\mathbb{N}}$,
specifying the actions $(a_0,a_1,\ldots)$ that should be taken, when the
system is in the state $x$. The goal is to find a policy $\pi$
maximizing the expected discounted sum, over an infinite horizon
\[
\mathbb{E}[\sum^{\infty}_{k=0} {\lambda^k r(x_k, a_k)}],
\]
where the expectation is taken with respect to the law of
the Markov chain defined by the above transition rates, and
$\lambda$ is the discount factor
satisfying $0 < \lambda < 1$ and $x_{k+1}=\phi_{a_{k}} x_k$, $x_0 =x$.
When there exists a solution
$V(x)=\max_{\pi}\mathbb{E}[\sum^{\infty}_{k=0} {\lambda^k c(x_k, y_k)}]$
for this problem it satisfies the stochastic discounted Bellman equation
\[
V(x)
=
\max_{a\in A} \left\{c(x, y) + \lambda \mathbb{E}[V(\phi_{a}x)] \right\}.
\]
For a comprehensive survey on MDP, see \cite{MR1270015}.

\subsubsection*{Discrete Aubry-Mather Problem}
In Lagrangian Mechanics, the Aubry-Mather problem \cite{MR1166538,MR1384478}
consists in finding probability measures defined on
the tangent fiber bundle $TM$ of a
manifold $M$ that minimizes the action of a convex and
superlinear Lagrangian $L: TM \to \mathbb{R}$, of class $C^2$,
that is,
\[
\inf_{\mu} \int_{TM} L(x,v)\, d\mu(x,v).
\]

In \cite{MR2128794} the author considers the case
$M=\mathbb{T}^{n}$, the $n$-dimensional torus and
the dynamics
$f: \mathbb{T}^{n} \times \mathbb{R}^{n} \to \mathbb{T}^{n}$
given by $f(x,v)=x+v$.
Define the discrete differential operator with respect to
$f$, acting on a function $g\in C(\mathbb{T}^{n},\mathbb{R})$ as follows
$d_{x}g (v):=g(f(x,v))- g(x)$. The minimization
is taken over the set of holonomic probability measures
\[
\mathcal{H}
:=
\left\{
	\mu\in \mathscr{P}(TM)
	\left|
		\
		\int_{TM}d_{x}g (v)\, d\mu(x,v)=0,
		\quad 	
		\forall g\in C(\mathbb{T}^{n},\mathbb{R})
	\right.
\right\}.
\]
By Fenchel-Rockafellar duality theorem, see \cite{MR0187062} and \cite{MR2128794}, we have
\[
-\inf_{\mu\in \mathcal{H}} \int_{TM} L(x,v)\, d\mu(x,v)
=
\inf_{g \in C(\mathbb{T}^{n},\mathbb{R})}
\sup_{(x,v)\in \mathbb{T}^{n} \times \mathbb{R}^{n} }
- d_{x}g (v) - L(x,v).
\]
This problem is related to one of finding the solutions of the discrete
Hamilton-Jacobi-Bellman equation
\[
\overline{H}=\sup_{v \in \mathbb{R}^{n}} -d_{x}g (v) - L(x,v),
\]
commonly solved by using viscosity solutions methods,
which is a dynamic programming problem associated to Bellman's operator
\[
T_{\alpha}(u)=\inf_{v \in \mathbb{R}^{n}} e^{-\alpha}u(f(x,v)) + L(x,v),
\]
for $\alpha >0$.
This operator defines a uniform contraction on a suitable Banach space
and its unique fixed point is the unique viscosity solution of the
Bellman's equation
\[
u_{\alpha}(x)
=
\inf_{v \in \mathbb{R}^{n}} e^{-\alpha}u_{\alpha}(f(x,v)) + L(x,v).
\]

In \cite{MR2128794} it is shown that
$u_{\alpha}(x) -\min u_{\alpha} \to u(x)$ and
$(1-e^{-\alpha})\min u_{\alpha} \to \overline{H}$,
when $\alpha \to 0$, and furthermore it is shown that
the limit function $u$ satisfies the equation
$u(x) = \inf_{v \in \mathbb{R}^{n}} u(x + v) + L(x, v) + \overline{H}$,
that is,
\[
\overline{H}
=
\sup_{v \in \mathbb{R}^{n}} u(x) - u(f(x,v)) - L(x,v)
=
\sup_{v \in \mathbb{R}^{n}} -d_{x}u (v) - L(x,v)
=
H(x, d_x u),
\]
where the Hamiltonian $H$ is the Legendre transform of $-L$.
Actually, in \cite{MR2128794}, the discount is
$T_{\alpha}(u)=e^{-\alpha}\inf_{v \in \mathbb{R}^{n}} u(f(x,v)) + L(x,v)$,
but the reasoning is exactly the same in both cases.

\subsubsection*{SRB-measures and Fat Solenoidal Attractors}
Sums controlled by IFS are also used to characterize the
boundary of attractors, of certain skew maps, and
to show when the SRB-measures are absolutely continuous. We recall that a skew map is a map $F: X\times Y \to X \times Y$ of the form $F(x,y)=(F_1(x), F_2(x,y))$, where $F_1$ is a self-map of $X$.
In \cite{MR1862809}, the author study the attractor of
the map $F: \mathbb{S}^1 \times \mathbb{R} \to \mathbb{S}^1 \times \mathbb{R}$ given by
$F(x,y) = ( T(x) , \lambda \, y + f(x))$,
where $T(x)=2x \mod 1$, $y \in \mathbb{R}$
and $f:\mathbb{S}^1 \to \mathbb{R}$ is a $C^2$ potential.

For a fixed $a=(a_0, a_1,  \ldots ) \in \{0,1\}^{\mathbb{N}}$
define
$
\phi_{k,a} x
:=
\phi_{a_{k-1}} \circ \phi_{a_{k-2}}\circ\ldots\circ\phi_{a_0}x
$,
where $\phi_i$, $i=0,1$,
are the inverse branches of $T(x)=2x\mod 1$.
A straightforward computation shows that for any $n\in\mathbb{N}$ we have
\[
F^{n}(\phi_{n,a} x, y)
=(x, \lambda^{n} y + \lambda^{0} f(\phi_{a_0} x)
+ \cdots +
\lambda^{n} f(\phi_{a_{n}}\cdots\phi_{a_0}  x)).
\]
The expression in rhs above lead us naturally
to consider the discounted controlled sums
given by $S(x,a):=\sum \lambda^{k} f(\phi_{k,a} x).$ 
In \cite{MR1862809} (see also \cite{MR2201152}
for topological properties of the attractor)
the author gives a description of the SRB measure,
by analyzing  $S(x,a)$ and conjectured
that the optimal return function $\sup_a S(x,a)$
can be used to describe the boundary of the attractor.
This conjecture was partially solved in \cite{2014arXiv1402.7313L},
assuming that the potential $f$ satisfies a certain twist condition.
A natural question arises when we change the skew map $F$ by a
non uniform hyperbolic one, with variable discount such as
$G(x,y)= (T(x) , \ln(1+y) + f(x))$ ( note that $\{1,2\}$ is always contained
in the spectrum of $DG(x,0)$ ).
This situation requires a variable discounted dynamic programming approach.

\subsubsection*{Ergodic Optimization}

A central problem in ergodic optimization consists in
finding an optimal invariant measure attaining the supremum
\[
m=
\sup_{\mu\in \mathscr{P}_{T}(X)}
\int_{X} f\, d\mu,
\]
where $(X, d)$ is a metric space, $T:X \to X$ is a continuous transformation
and $f: X \to \mathbb{R}$ is a given potential.

For example, in case where $X=\mathbb{R}/\mathbb{Z}$ and the transformation
$T:X \to X$ is the double mapping,
the ergodic optimization problem can be viewed as a decision problem
for IFS as follows. We take the IFS $\{\phi_0,\phi_1\}$,
where $\phi_{0}x=(1/2)x$ and $\phi_{1}x=(1/2)x +1/2$,
the set of possible actions is $A= \{0, 1\}$ and the immediate return
$c(x,a):= f(\phi_{a} x)$.

Under fairly general conditions on the potential $f$,
we can prove several theorems about the
support of maximizing measures.
For example, the solutions of Bellman's equation
\[
b(x) = \max_{a\in A} \{f(\phi_{a} x)  + \lambda \, b(\phi_{a} x )\},
\]
can be characterized if the potential $f$ satisfies a twist condition.
By taking the limit when $\lambda \to 1$, we obtain a subaction $V$ satisfying
\[
V(x) =  \max_{a\in A} \{f(\phi_{a} x)- m + \, V(\phi_{a} x )\}.
\]
The support of a maximizing measure $\nu$, notation $\supp \nu$,
is contained in the set where we have the equality in the above expression,
see \cite{MR1841880,MR3701349} and
the recent survey \cite{ETDSjenkinson_2018}.

\section{Sequential Decision-Making Processes}

In this section we introduce very general setting to
handle some variational problems in DP.
The applications discussed here will be obtained
by considering additional regularity conditions and
specializing the spaces, functions and so on.
Our starting point will be the following definition.

\begin{definition}[Sequential Decision-Making Process]\label{def-DMP}
A sequential decision-making process is an ordered sextuple
$S=\{X, A, \Psi, f, u, \delta\}$, where
	\begin{itemize}[itemsep=0pt, topsep=1pt, partopsep=0pt]
	\item $X$ is a complete metric space, called state space;
	
	\item $A$ is a general metric space, called set of all available actions;
	
	\item $\Psi: X \to 2^A$ is a set-valued function. For all $x\in X$
	the set $\Psi(x) \subseteq A$ is always assumed to be a non-empty compact set
	and called the set of all feasible actions for an agent $x$.
	We shall assume that $\Psi$ is	continuous, with respect to
	the Hausdorff topology on the not-empty compact subsets of $A$.
	
	\item $f:X \times A \to X$ is a continuous map, called transition law for the system;
	
	\item $u:X \times A \to \mathbb{R}$ is continuous function and
	$u(x,a)$ is called the immediate reward or return associated with taking
	the action $a$ in the state $x$;
	
	\item $\delta:D\subset \mathbb{R} \to \mathbb{R}$,
	is an increasing continuous function called discount function.
	It represents the relevance of taking an action at the next step.
	\end{itemize}
\end{definition}

Although linear discount, by a factor $\beta \in (0,1)$, can be employed to solve
several problems in DP, it may not be a suitable tool to handle some other complicated problems.
A natural alternative would be consider variable discount factor or even a variable discount function.
In order to give a precise definition of this concept,
let us introduce the notion of a generalized modulus of contraction.

\subsection{Variable Discount Functions}

\begin{definition} \label{witness function}
A generalized modulus of contraction
for a function $\delta:D\subset \mathbb{R} \to \mathbb{R}$ is
an increasing function $\gamma:[0,\infty) \to [0,\infty)$
such that for all $t\geq 0$ the $n$-th iterate $\gamma^n(t)\to 0$, when $n\to\infty$, and
\[
|\delta(t_2) - \delta(t_1)| \leq \gamma(|t_2 - t_1|)
\]
for any $t_1, t_2 \in D$.
\end{definition}

Any function $\gamma$ as above satisfies $\gamma(0)=0$.
Indeed, if $\gamma(0)=\gamma_0 >0$
the monotonicity of $\gamma$ implies that
$\liminf \gamma^n(t)\geq \gamma_0 >0$.
By using a similar reasoning, we can prove that $\gamma(t)< t$,
for all $t>0$.

\begin{definition}[Variable discount function]\label{discount function}
A function $\delta:D\subset \mathbb{R} \to \mathbb{R}$
will be called a variable discount function if it has a generalized modulus
of contraction  $\gamma:[0,\infty) \to [0,\infty)$.
A variable discount function $\delta$ is called
\begin{itemize}[itemsep=0pt, topsep=7pt, partopsep=5pt]
\item[a)] idempotent if $\delta=\gamma$, for some generalized modulus
of contraction $\gamma$;
\item[b)] subadditive if
$\delta(t_1 + t_2) \leq \delta(t_1 ) + \delta( t_2)$,
for any $t_1, t_2 \in D$ such that $t_1+t_2\in D$.
\end{itemize}
\end{definition}

\begin{proposition}
Let $\delta:[0,\infty) \to [0,\infty)$ be a continuous and increasing function
satisfying $\delta^n(t)\to 0$, when $n\to\infty$.
If $\delta$ is subadditive then it is idempotent.
\end{proposition}
\begin{proof}
From subadditivity we get, for any pair $x,y\in [0,\infty)$ satisfying $x\geq y$,
the following inequality
\[
\delta(x) -\delta(y)
=
\delta(x-y  + y) -\delta(y) \leq \delta(x-y  )
+ \delta( y) -\delta(y)\leq \delta(x-y).
\]
Similarly, we obtain $\delta(y) -\delta(x)\leq \delta(y-x)$, for $y\geq x$.
Since $\delta$ is an increasing function we get that
$|\delta(x) -\delta(y)|\leq \delta(|x-y|)$.
By taking $\gamma=\delta$, it follows from
the hypothesis that $\delta$ is itself
a modulus of contraction for $\delta$.
\end{proof}

\begin{example}\label{ex1discfunc}
Given $\beta \in(0,1)$, the function $\delta(t):=\beta t$
for $t \in \mathbb{R}$ is a idempotent discount function, because it is linear.
This is the canonical discount function used in dynamic programming.
\end{example}

\begin{example}\label{ex2discfunc}
The function $\delta(t):=\ln(1+t)$ for $t \geq 0$
is a nonlinear idempotent discounted function.
Indeed,
\begin{align*}
|\delta(t_1) - \delta(t_2)|
\leq
|\ln\left(\frac{1+t_1}{1+t_2}\right)|
\leq
\ln\left(1+ \frac{|t_1-t_2|}{1+t_2}\right)
\leq
\ln(1+|t_1-t_2|)
\end{align*}
and $0<\gamma'(t)=1/(1+t) <1$, for all $t>0$ so
$\gamma^n(t)\to 0$, when $n\to\infty$.
Therefore $\gamma(t)=\delta(t)$
is a generalized modulus of contraction.
Note that $\delta$ is also subadditive. Indeed,
$
\delta(t_1 + t_2)
\leq
\ln(1+(t_1 + t_2))
\leq
\ln(1+(t_1 + t_2)+ (t_1 \cdot t_2))
=
\ln((1+t_1)(1 + t_2))
=
\delta(t_1 ) + \delta( t_2).
$
\end{example}

\begin{example}\label{ex3discfunc}
A piecewise linear function $\delta_1: \mathbb{R} \to \mathbb{R}$ defined by
\[
\delta_1(t):=\left\{
    \begin{array}{ll}
      \beta t, & t \leq 1;\\
      \frac{\beta}{2} t + \frac{\beta}{2}, & t >1,
    \end{array}
  \right.
\]
where $\beta \in (0,1)$ is also a
variable discount function with the same generalized
contraction modulus as $\delta(t):=\beta t$.
Additionally, $\delta_1$ is an example of subadditive but
not idempotent variable discount function.
\end{example}

\begin{example}\label{ex4discfunc}
Consider the function $\delta: [0,\infty) \to [0,\infty)$ defined by
\[
\delta(t)=-1+\sqrt{t+1}.
\]
We have that $\delta(0)=0$ and $\delta$ is an increasing function and
for all $t_1,t_2\in [0,\infty)$, we have
\[
|\delta(t_1) - \delta(t_2)|
=
\left|\sqrt{t_1 +1} - \sqrt{t_2 +1}\right|
=
\left|\frac{(t_1 +1) - (t_2 +1)}{\sqrt{t_1 +1} + \sqrt{t_2 +1}}\right|
\leq
\frac{1}{2}|t_1 - t_2|.
\]
By taking $\beta=1/2 \in (0,1)$, we can show that
$\delta$ is a variable discount function
with the same generalized contraction modulus as
$\gamma(t):=\beta t$.
\end{example}

\begin{example}\label{ex5discfunc}
If $\delta: [0,\infty) \to [0,\infty)$ is a $C^{2}$-function such that:
\begin{itemize}
  \item[a)] $\delta(0)=0$;

  \item[b)] $\delta'(0^{+}):=\lim_{t\downarrow 0}\delta'(t)=\beta \in (0,1)$
  and $\delta'(t) > 0$;

  \item[c)] $\delta''(t) \leq 0$.
\end{itemize}
Then $\delta$ is increasing and
$|\delta(t_1) - \delta(t_2)| = \delta'(t_0)|t_1 - t_2| \leq \delta'(0^{+}) |t_1 - t_2|$,
since $\delta'(t_0) \leq \delta'(0^{+})$.
Thus $\gamma(t):=\beta t$ is a generalized contraction modulus for $\delta$.
Note that for every fixed $p > 1$ the function
$\delta(t)
=
-1+(t+1)^{1/p},
$
satisfies conditions \textit{a)}--\textit{c)} with  $\delta'(0^{+})= 1/p$.
This generalizes the Example~\ref{ex4discfunc}, when $p=2$.
\end{example}

The main reason to consider such general variable discounts is
to develop a perturbation theory.
The idea is to consider a parametric family of discounts
$\delta_{n}:[0,+\infty) \to \mathbb{R}$, where $\delta_{n}(t) \to I(t)=t$,
in the pointwise topology, and then
to study the properties of possible limits, when $n\to\infty$,
of the fixed points $v_{n}(x)$ and $w_{n}(x)$.
In this regard, we consider sequences of variable discount decision-making
process $(S_n)_{n\in\mathbb{N}}$, where $S_n=\{X, A, \Psi, f, u, \delta_{n}\}$ is
defined by a continuous and bounded immediate reward
$u:X \times A \to \mathbb{R}$ and sequence of discounts
$(\delta_n)_{n\geq0}$, satisfy some admissibility conditions:
\begin{itemize}
  \item[\textit{a)}] the contraction modulus  $\gamma_{n}$ of the variable discount
  $\delta_{n}$ is also a variable discount function;

  \item[\textit{b)}] $\delta_{n}(0)=0$ and $\delta_{n}(t) \leq t$, for $t\geq 0$.

  \item[\textit{c)}]  $\delta_{n}(t +\alpha) - \delta_{n}(t)\to \alpha $,
  when $n\to\infty$, uniformly in $t>0$, for any fixed constant $\alpha \geq 0$.
\end{itemize}
In Section \ref{section-discounted-limits} we prove
two of the main results of this paper which are Theorem~\ref{bellman subaction},
ensuring the existence of a value $\bar{u} \in [0 , \| u \|_{\infty} ]$ and a function $h$ such that
\[
h(x)= \max_{a \in \Psi(x)} u(x, a)- \bar{u} +  h(f(x,a)),
\]
and Theorem~\ref{ruelle theorem discounted} which guarantees the existence of
a value $k \in [0 , \| u \|_{\infty} ]$ and a function $h$
given by
\[
h(x)=\ln \int_{a \in \Psi(x)} e^{u(x, a)+h(f(x,a)) - k}
         \, d\nu_{x}(a),
\]
such that  $\rho:=e^{k}$  and $\varphi:= e^{h(x)}$ are the maximal
eigenvalue and eigenfunction of the Ruelle operator, that is,
\[
e^{k} e^{h(x)}=\int_{a \in \Psi(x)} e^{u(x, a)} e^{h(f(x,a))} d\nu_{x}(a).
\]
For both results the key hypothesis in $u$
are uniformly $\delta$-boundedness and uniformly $\delta$-domination,
see Definition \ref{gamma continuity}.

Regarding this hypothesis on $u$, we want to stress that we prove in
Theorem~\ref{regularity implies domination} that if $f$ is a contractive
dynamics , that is,
\[
\sup_{a \in A} d_{X}(f(x,a),f(y,a)) \leq \lambda d_{X}(x,y)
\]
and $u(\cdot, a)$ is $C$-Lipschitz (or $\alpha$-H\"older)
then $u$ is uniformly $\delta$-dominated.
If additionally, $\mathrm{diam}(X) < \infty$, then $u$ is uniformly $\delta$-bounded.
In particular, if $v_n$ and $w_n$ are respectively the solutions of
Bellman's equation and the transfer discounted operator equation,
they are uniformly $C(1-\lambda)^{-1}$-Lipschitz
(or $\alpha$-H\"older, with
$
{\rm Hol}_{\alpha}(v_n)={\rm Hol}_{\alpha}(w_n)
= {\rm Hol}_{\alpha}(u) (1-\lambda^{\alpha})^{-1}
$).
This shows that most of the previous results in the literature for IFS or
expanding maps, with either Lipschitz or H\"older weights are particular cases of our theorems,
with constant discounts satisfying
$\delta_{n}(t)=\beta_n t$, where $0<\beta_n <1$ and $\beta_n \to 1$.

\subsection{Generalized Matkowski Contraction Theorem}
In 1975, Janusz Matkowski \cite{MR0412650},
obtained a generalization of Banach's contraction
theorem for a variable contraction map. Before state
this result we need one more definition.

\begin{definition} \label{witness function metric space}
Let $(X,d)$ be a complete metric space and $T:X \to X$ a map.
We say that $T$ is a generalized Matkowski contraction,
if there exists a witness function for $T$, that is,
a non-decreasing function $\varphi:[0,\infty) \to [0,\infty)$
such that $\varphi^n(t)\to 0$, when $n\to\infty$ and
\[
d(T(x),T(y))
\leq
\varphi(d(x,y))
\]
for any $x,y \in X$.
\end{definition}

When the contraction is not fixed, e.g. $d(T(x),T(y))\leq \lambda d(x,y)$, 
the function $\varphi$ witness the fact that $T$ is a generalized 
contraction, e.g. $d(T(x),T(y))\leq \varphi(d(x,y))$. 
In other words, it is not enough to say that $T$ is a generalized contraction, 
we need a witness $\varphi$.

\begin{theorem}[\cite{MR0412650}]	
\label{Matkowski Contrac Theorem}
If $(X,d)$ is a complete metric space and
$T:X \to X$ a generalized Matkowski contraction,
then there exists a unique $x_0 \in X$ such that
$T(x_0)=x_0$ and $d(T^n(x), x_0)\to 0$ for all $x_0 \in X$.
\end{theorem}

A weaker version of this theorem was known. It
required the witness function $\varphi$ to be right USC
(instead of non-decreasing) and $\varphi(t)<t$, for all
$t>0$ (instead of $\varphi^n(t)\to 0$, when $n\to\infty$).
However, the set of all contractions where
Theorem~\ref{Matkowski Contrac Theorem} works is wider than this one,
as pointed by Matkowski,
we may apply the theorem for a map $T$, having a witness function
\[
\varphi(t)
:=
\begin{cases}
1 &, t>1;\\
\frac{1}{n+1} &, \frac{1}{n+1}< t \leq\frac{1}{n};\\
0 &, t=0,
\end{cases}
\]
which is not a right USC function.

\subsection{Variable Discounting in Dynamic Programming}\label{dyn prog section}
This section is devoted to present some results of the recent theory
developed by Ja\'{s}kiewicz, Matkowski and Nowak \cite{MR3029480,MR3200713,MR3248091}.
The applications in these works focused on Markov decision processes, and
the theory of optimal economic growth and resource extraction models, but as will be explained below
it has far-reaching consequences.

We shall consider a sequential decision-making process
$S=\{X, A, \Psi, f, u, \delta\}$ as a dynamical system specified as follows:
at the state $x_0$ we take an action $a_0$, and receive an
immediate return $u(x_0,a_0)$ and go forward to the new state
$x_1=f(x_0, a_0)$. Based on it, one decides to take a new action
$a_1 \in \Psi(x_{1})$ and so on.
In this way we obtain a feasible sequence
$(x_0,a_0, x_1, a_1,\ldots) \in (X \times A)^{\mathbb{N}} $
which is a orbit of the dynamical system $S$.

\begin{definition}
The set of all the feasible sequences  of a sequential
decision-making process $S=\{X, A, \Psi, f, u, \delta\}$ is
given by
\[
\Omega
:=
\{
(x_0,a_0, x_1, a_1,\ldots) \in (X \times A)^{\mathbb{N}} |
\; x_{i+1}=f(x_{i}, a_{i}), \;  a_{i} \in \Psi(x_{i})
\}.
\]
\end{definition}

Typically, the above defined set is strictly
contained in the Cartesian pro\-duct, that is,
$\Omega \subsetneq (X \times A)^{\mathbb{N}}$,
unless $\Psi(x)=A$, for all $x \in X$ and $f$ is surjective.
It is useful to define the set of all feasible action sequences
starting from $x_0$,
\[
\Pi(x_0)=\{\bar{a}=(a_i) \in A^\mathbb{N} \; | \; (x_0,a_0, x_1, a_1,\ldots) \in \Omega \}.
\]

We point out that an element in $\Omega$ depends only on the initial point
$x_0$ and on a feasible action sequence $\bar{a}\in \Pi(x_0)$,
so we can use a concise notation:
\[
h_{x_0}(\bar{a})=(x_0,a_0, x_1, a_1,\ldots) \in \Omega.
\]
\begin{proposition} \label{feasible topology}
The set $\Omega \subset (X \times A)^{\mathbb{N}} $
is closed relative to the product topology on $(X \times A )^{\mathbb{N}}$.
\end{proposition}
\begin{proof}
Since $X$ is complete and $\Psi(x)$ is compact, for all $x\in X$,
we can obtain, by an inductive argument, a feasible sequence in $\Pi(\lim_i x_0^i)$
for any Cauchy sequence $(h_{x_0^i }(\bar{a}^i))_{i\geq 0}$.
\end{proof}

\begin{remark}
An alternative  way to define
the space $\Omega$ is to introduce it as the set $\Omega':=\{(x, \Pi(x)) \; | \; x \in X\}$.
The set $\Omega'$ is like a fiber bundle and has a natural structure of metric space
\[
d_{\Omega'}((x,\bar{a}), (y,\bar{b})):= d_{X}(x,y) +d_{A}(\bar{a}, \bar{b})
\]
Thus $(\Omega', d_{\Omega'})$ is a complete metric space and the
topology is equivalent to the product topology.
\end{remark}

\begin{definition}\label{rec ut}
Let $u: X \times A \to \mathbb{R}$ be a bounded from above function.
A recursive utility  associated to the immediate rewards $u(x,a)$
with a discount function $\delta$ is a function $U:\Omega \to \mathbb{R}\cup\{-\infty\}$, such that
\[
U(h_{x_0}(\bar{a}))=u(x_0, a_0) +
\delta(U(h_{x_1}\sigma\bar{a}))
\]
for any  history $h_{x_0}(\bar{a})$.
\end{definition}

\begin{definition}\label{induc lim}
Let $u(x,a)$ be an immediate rewards and $\delta$
a discount function. We define,
for any  history $h_{x_0}(\bar{a})$, the associated inductive limit
\[
{\sum_{i}}^{\star}
u(x_i, a_i)
=
\displaystyle\lim_{n \to \infty}
{\sum_{i\in [n]} }^{\star}
u(x_i, a_i),
\]
where
\[
{\sum_{i\in [n]} }^{\star}u(x_i, a_i)
=u(x_0, a_0) + \delta\big(  u(x_1, a_1) + \delta(u(x_2, a_2)+\ldots+\delta(u(x_n,a_n) )\big)
\]
and the notation $[n]$ stands for the interval $\{0,1,\ldots,n\}$ in the set of integers numbers.
\end{definition}
Corollary~\ref{convergence inductive limits} provides
necessary conditions to ensure the existence of the above limit.

The connection between Definitions \ref{rec ut} and \ref{induc lim} is given by the next proposition.
\begin{proposition}
  Let $u(x,a)$ be a bounded from above immediate rewards and $\delta$
a continuous discount function. If
$\sum^{\star}_{i} u(x_i, a_i)$
converges then the function $V:\Omega \to \mathbb{R}\cup\{-\infty\}$
defined by $V(h_{x_0}(\bar{a}))= \sum^{\star}_{i} u(x_i, a_i)$
is a recursive utility.
Reciprocally, if $U:\Omega \to \mathbb{R}\cup\{-\infty\}$
is a bounded from above recursive utility then $U$
is represented by $U(h_{x_0}(\bar{a}))= \sum^{\star}_{i} u(x_i, a_i)$.
\end{proposition}
\begin{proof}
For the first part we define $V(h_{x_0}(\bar{a}))= \sum^{\star}_{i} u(x_i, a_i)$.
A simple computation shows that
\begin{align*}
{\sum_{i\in [n]} }^{\star}u(x_i, a_i)
&=
u(x_0, a_0) +  \delta\Big( {\sum_{i\in [n-1]} }^{\star}u(x_{i+1}, a_{i+1})\Big).
\end{align*}
Using the continuity of $\delta$ and taking the limit we obtain \[
V(h_{x_0}(\bar{a}))=u(x_0, a_0) +
\delta(V(h_{x_1}\sigma\bar{a})).
\] Therefore, $V$ is a recursive utility.

Reciprocally, if $U$ is a bounded from above recursive utility,
then we have $U \leq K$ for some $K>0$ and
\begin{align*}
U(h_{x_0}(a))
&=
u(x_0, a_0) + \delta(U(h_{x_1}\sigma\bar{a}))
\\
&=
u(x_0, a_0) + \delta(u(x_1, a_1) + \delta(U(h_{x_2}\sigma^2\bar{a})))
\\
&=
u(x_0, a_0) +  \delta( u(x_1, a_1) + \delta(\cdots u(x_{n-1}, a_{n-1})
+ \delta(U(h_{x_n}\sigma^n\bar{a}))) ).
\end{align*}
By using repeatedly the inequality
$|\delta(t_2) - \delta(t_1)| \leq \gamma(|t_2 - t_1|)$ we have
\begin{align*}
|{\sum_{i\in [n]} }^{\star}u(x_i, a_i) - U(h_{x_0}(a))|
\leq
\gamma^n(U(h_{x_n}\sigma^n\bar{a}))
\leq
\gamma^n(K)\to 0,
\end{align*}
proving that $U(h_{x_0}(\bar{a}))= \sum^{\star}_{i} u(x_i, a_i)$.
\end{proof}

\begin{definition} \label{optimal return}
Given $U: \Omega \to \mathbb{R}$ a
function, $\hat{V}(x)=\sup_{h_x(\bar{a})\in\Omega} U(h_x(\bar{a}))$
is called an optimal return. An element $a^* \in \Pi(x)$  (sometimes called plan)
is said to be optimal if $\hat{V}(x)= U(h_x(a^*))$.
\end{definition}

\begin{definition} \label{standard aggregator}
A function $W: X\times A \times D  \to \mathbb{R}$ given by
\[
W(x,a,r):=u(x,a)+ \delta(r),
\]
where  $a \in \Psi(x)$ is called an aggregator function.
\end{definition}

In dynamic programming we can always assume that $\delta(0)=0$,
otherwise we can redefine $\tilde{u}(x,a)=u(x,a)+\delta(0)$ and
$\tilde{\delta}(t)=\delta(t) -\delta(0)$ without changing the
aggregator function value neither the solutions of some problems associated
to it.

Now we introduce some dynamics on $\Omega$, by considering
the maps
\begin{itemize}
\item[\textit{a)}]
$\sigma: \Pi(x) \to \Pi(f(x,\cdot))$ the left shift
given by $\sigma(a_0, a_1,\ldots)=(a_1, a_2,\ldots)$.
Note that this mapping is well-defined since for any $(a_0, a_1,\ldots) \in \Pi(x)$ we have
that $(a_1, a_2,\ldots) \in \Pi(f(x,a_0))$;

\item[\textit{b)}]
$\phi: X \times \Pi(\cdot)  \to X$ the skew map
\[\phi_{a} x = f(x,a), \; a \in \Psi(x);\]

\item[\textit{c)}]
$\hat{\sigma}: \Omega \to \Omega$ the ``double left shift'' operator given by
\[
\hat{\sigma}(h_x(\bar{a}))
:=
h_{\phi_{a_{0}}(x)}(\sigma(\bar{a}))
:=
(x_1,a_1, x_2, a_2,\ldots) \in \Omega.
\]
\end{itemize}

\begin{definition} \label{koopman operator}
Given a bounded and continuous
immediate reward $u$ and a variable discount function $\delta$,
satisfying $\delta(0)=0$, the Koopman operator
$K:=K_{u,\delta}:C_{b}(\Omega,\mathbb{R})\to C_{b}(\Omega,\mathbb{R})$
is defined by
\[
K(U)(h_x(\bar{a}))
=
W(x_0, a_0, U(\hat{\sigma}(h_x(\bar{a}))).
\]
\end{definition}

Note that a fixed point for the Koopman operator, that is, $K(U)=U$ is
a recursive utility, in the sense of Definition \ref{rec ut}.

\begin{theorem}[\cite{MR3248091}]
\label{recursive utility existence and uniqueness}
Let $u$ be a bounded and continuous immediate reward, and $\delta$ a variable discount,
satisfying $\delta(0)=0$. Then there exists a unique fixed point $U \in C_{b}(\Omega, \mathbb{R})$,
for the Koopman operator and
moreover
\[
\|K^n (Q) - U\|_{\infty} \to 0
\]
for any $Q \in C_{b}(\Omega, \mathbb{R})$.
\end{theorem}
\begin{proof}
Since the function $u, \delta$ and $f$ are continuous and $u$ is bounded
we have that $K (C_{b}(\Omega, \mathbb{R})) \subseteq C_{b}(\Omega, \mathbb{R})$.
The result is a consequence of Theorem~\ref{Matkowski Contrac Theorem}
because $K$ is a generalized Matkowski contraction with the witness
function $\varphi(t):= \gamma (t)$, where $\gamma$ is the
contraction modulus of $\delta$ and the metric space
$(C_{b}(\Omega, \mathbb{R}), \|\cdot\|_{\infty})$ is complete.
\end{proof}

As a corollary we obtain sufficient conditions for the
existence of the inductive limits.
\begin{corollary}\label{convergence inductive limits}
Under the assumptions of Theorem\ref{recursive utility existence and uniqueness}
there exists the inductive limit
\[
{\sum_{i}}^{\star}
u(x_i, a_i)
=
\displaystyle\lim_{n \to \infty}
{\sum_{i\in [n]} }^{\star}
u(x_i, a_i),
\]
where the convergence is in the uniform topology.
In particular,
\[
{\sum_{i}}^{\star}
u(x_i, a_i)
=
\lim_{n \to \infty}K^n (0)(h_x(\bar{a}))
=
U(h_x(\bar{a})),
\]
is the unique bounded continuous recursive utility.
\end{corollary}

\section{Bellman and Discounted Transfer Operators}
Note that until now, we have only assumed that $u$ is a bounded
continuous function.
In the sequel, we add an extra assumption which is $u\geq 0$.
This technical assumption is convenient when considering iterates of $K$,
since $\delta$ is only defined on $D:=[0, +\infty)$.
This is actually not a restrictive assumption since in
the bounded continuous case, we can always replace $u$ by $u-\min u \geq 0$.
See Remark~\ref{add cte by u dont change} for further details
on this issue.

\begin{definition} \label{Bellman operator}
Given a non-negative bounded and continuous
immediate reward $u$ and a discount function $\delta$,
satisfying $\delta(0)=0$,
the Bellman operator $B:=B_{u,\delta}: C_{b}(X, \mathbb{R})\to C_{b}(X, \mathbb{R})$
applied to $v$ and evaluated at $x$ is defined by
\[
B(v)(x)
:=
\sup_{a \in \Psi(x)} W(x, a, v(f(x,a))).
\]
\end{definition}
\begin{definition} \label{discounted transfer operator}
Let $u$ and $\delta$ be as in Definition \ref{Bellman operator}.
The discounted transfer operator,
$P:=P_{u,\delta}:C_{b}(X, \mathbb{R})\to C_{b}(X, \mathbb{R})$,
applied to $v$ and evaluated at $x$ is defined by
\begin{align*}
P(v)(x)
&:=
\ln \int_{a \in \Psi(x)} e^{W(x, a, v(f(x,a)))} d\nu_{x}(a),
\end{align*}
where $\nu_{x}$ is a Borel probability measure on $A$,
satisfying $\nu_{x}(\Psi(x))=1$, for all $x\in X$.

The transfer operator, or Ruelle operator,
is the linear operator on $C_{b}(X,\mathbb{R})$ defined by
\[
L(v)(x)
:=
\int_{a \in \Psi(x)} e^{u(x, a)} v(f(x,a)) d\nu_{x}(a).
\]
\end{definition}

Before proceed, we recall a basic fact from general topology.
For more details, see reference \cite{MR1464690}, page 115, Theorems 1 and 2.

\begin{theorem}\label{maximum theorem}
Let $X, Y$ be topological spaces $F:X \times Y \to \mathbb{R}$ a
USC (resp. LSC) mapping and $\Gamma: X \to 2^Y$ a USC
(resp. LSC) set valued map, such that
$\Gamma(x) \neq \varnothing$, for all $x \in X$. Then the function
\[
M(x)
:=
\sup_{y \in \Gamma(x)} F(x,y),
\]
is a USC (resp. LSC). In particular,
if $F$ and $\Gamma$ are continuous, then $M$ is continuous.
\end{theorem}

\begin{lemma} \label{continuity bellman op}
  The Bellman and discount transfer operators, defined above,
  send the space $C_{b}(X, \mathbb{R})$ to itself.
\end{lemma}
\begin{proof}
The prove that $B( C_{b}(X, \mathbb{R}) )\subset C_{b}(X, \mathbb{R})$ it is
enough to apply Theorem~\ref{maximum theorem},
with $X=X$, $Y=A$, $\Gamma=\Psi$ and $F(a):=u(x, a)+\delta(v(f(x,a))$,
which is clearly continuous, thus showing that
\[
B(v)(x):=\sup_{a \in \Psi(x)} F(a)=\max_{a \in \Psi(x)} u(x, a)+\delta(v(f(x,a))
\]
is a continuous and bounded function.

For the discount transfer operator the proof is similar. We keep the
above setting and consider the continuous functions
\[
M(x):=\sup_{a \in \Psi(x)} F(a)
\quad\text{and}\quad
N(x):=\inf_{a \in \Psi(x)} F(a).
\]

From definition of $P$, we have $N(x)\leq P(v)(x) \leq M(x)$,
for all $x\in X$. Therefore,
$
-(M(y)-N(x))
\leq
P(v)(x) - P(v)(y)
\leq
M(x)-N(y)
$
and the continuity and boundedness of $M$ and $N$ imply
that $x\longmapsto P(v)(x)$ is continuous and bounded function.
\end{proof}

\begin{theorem}[\cite{MR3248091}]
\label{max op solution}
Let $u$ and $\delta$ be as in Definition \ref{Bellman operator} and
$B:C_{b}(X, \mathbb{R}) \to C_{b}(X, \mathbb{R})$, the Bellman operator
associated to $u$ and $\delta$. Then
\begin{itemize}
\item[a)]
There is a unique $v^* \in C_{b}(X, \mathbb{R})$
such that $B(v^*)=v^*$. Moreover, $v^*$ is an optimal return and
satisfies the $\delta$-discounted Bellman equation
\[
v^*(x)
:=
\max_{a \in \Psi(x)} u(x, a)+\delta(v^*(f(x,a)).
\]

\item[b)]
A plan $a^* \in \Pi(x)$ attaining the maximum
$v^*(x_{n})= u(x_{n}, a_{n}^*)+\delta(v^*(x_{n+1}))$
for all $n \in \mathbb{N}$ is optimal.
In particular, there exists $a^* \in \Pi(x)$
such that $v^*(x):=U(h_x(a^*))$
\end{itemize}
\end{theorem}

\noindent \textit{Sketch of the proof.}
We provide here, for the reader's convenience, some of key steps of this proof.

\textit{a)} The existence of $v^{*}$ is a consequence of
Theorem~\ref{Matkowski Contrac Theorem}, because $B$ is a
generalized Matkowski contraction and the metric space
$(C_{b}(X, \mathbb{R}), \|\cdot\|_{\infty})$ is complete.
Indeed, one can show that
$\displaystyle \| B(v) -B(v')\|_{\infty} \leq \gamma (\| v -v'\|_{\infty})$.

\textit{b)} To show that $v^*$ is optimal, we consider any $h_{x}(\bar{a}) \in \Omega$.
From the fixed point equation we obtain
$v^*(x_{0})\geq u(x_{0}, a_{0}^*)+\delta(v^*(x_{1})),$
where $x_0=x$ and $x_1=f(x_0,a_0)$.
By iterating this equality we get
\begin{align*}
v^*(x_{0})
&\geq
u(x_{0}, a_{0}^*)+\delta( u(x_{1}, a_{1}^*)+\delta(v^*(x_{2})) ),
\\
v^*(x_{0})
&\geq
u(x_{0}, a_{0}^*)+\delta( u(x_{1}, a_{1}^*)+\delta( u(x_{2}, a_{2}^*)+\delta(v^*(x_{3}))  ) )
\end{align*}
and so on.
If $\zeta: \Omega \to \mathbb{R}$ is a function given by $\zeta(h_{x}\bar{a})=v^*(x)$,
then
$
K(\zeta)(h_{x}\bar{a})
=
u(x_{0}, a_{0}^*)+\delta(v^*(x_{1})),\ldots, K^3(\zeta)(h_{x}\bar{a})
=
u(x_{0}, a_{0}^*)+\delta( u(x_{1}, a_{1}^*)+
\delta( u(x_{2}, a_{2}^*)+\delta(v^*(x_{3}))  ) )
$, and so on. Therefore
$
v^*(x_{0})\geq K^n(\zeta)(h_{x}\bar{a}) \to U(h_x(\bar{a})),
$
where $U$ is the recursive utility given by the associated Koopman operator.
Thus showing that
\[
v^*(x) \geq \sup_{h_x(\bar{a})\in\Omega} U(h_x(\bar{a})).
\]

To show the equality, we use the continuity of $u$, $\delta$ and $v^*$.
The compactness of $\Psi(\cdot)$ allow us to choose,
from the fixed point equation, a sequence $a^* \in \Pi(x)$
attaining the maximum $v^*(x_{n})= u(x_{n}, a_{n}^*)+\delta(v^*(x_{n+1}))$,
for all $n \in \mathbb{N}$. Proceeding as before,
we obtain $v^*(x)=U(h_x(a^*))$.
So $v^*$ is optimal and there exists $a^* \in \Pi(x)$ such that
\[
v^*(x)
=
U(h_x(a^*))
=
\sup_{h_x(\bar{a})\in\Omega} U(h_x(\bar{a})).
\qedhere
\]

\begin{theorem}\label{disc transf op solution}
Let $u$ and $\delta$ be as in Definition \ref{Bellman operator}, and
$B:C_{b}(X, \mathbb{R}) \to C_{b}(X, \mathbb{R})$ the Bellman operator,
associated to this pair. Then
\begin{itemize}
\item[a)] there is a unique $w^* \in C_{b}(X, \mathbb{R})$ such that $P(w^*)=w^*$;

\item[b)] $w^*\leq v^*$ where $v^*$ is the unique solution of the Bellman equation
\[ v^*(x)= \max_{a \in \Psi(x)} u(x, a)+\delta(v^*(f(x,a)));\]

\item[c)] if the family of measures $\nu_{x}$ can be chosen in such
way that $\nu_{x}=\delta_{a_0}(x)$
where $a_0 \in {\rm argmax} \{u(x, a)+\delta(v^*(f(x,a)))\}$, then $w^*=v^*$.
\end{itemize}
\end{theorem}

\begin{proof}
\textit{a)} It is easy to see that
\begin{align*}
\|P(w_1) - P(w_2)\|_{\infty}
&\leq
\max_{a}|\delta(w_1(f(x,a))) - \delta(w_2(f(x,a)))|
\\
&\leq
\max_{a} \gamma(|w_1(f(x,a)) - w_2(f(x,a))|)
\\
&\leq
\gamma(\|w_1 - w_2\|_{\infty})
\end{align*}
and so $P$ is a generalized Matkowski contraction in the
complete metric space $(C_{b}(X, \mathbb{R}), \|\cdot\|_{\infty})$.
By Theorem~\ref{Matkowski Contrac Theorem} there is a unique
$w^* \in C_{b}(X, \mathbb{R})$ such that
$P(w^*)=w^*$ and $\|P^n (w) - w^*\|_{\infty} \to 0$, when $n\to\infty$,
for any $w \in C_{b}(X, \mathbb{R})$.

\medskip	
\textit{b)} To see that $w^*\leq v^*$ where $v^*$ is the unique
solution of the $\delta$- discounted Bellman equation
$v^*(x):= \max_{a \in \Psi(x)} u(x, a)+\delta(v^*(f(x,a))$,
we recall that
\begin{align*}
P(v^*)(x)
&=
\ln
	\int_{a \in \Psi(x)}
		\exp\big( u(x, a)+\delta(v^*(f(x,a))) \big)
	\, d\nu_{x}(a)
\\
&\leq
\ln
	\int_{a \in \Psi(x)}
		\exp\Big( \max_{a \in \Psi(x)} u(x,a)+\delta(v^*(f(x,a))) \Big)
	\, d\nu_{x}(a)
\\
&=
v^*(x).
\end{align*}
Since $\delta$ is an increasing function it follows that $P(v^*)\leq  v^*$, $P^2(v^*)\leq  v^*$ and so on.
Since $P^n(v^*) \to w^*$, when $n\to\infty$, we get from the previous inequality that
$w^*\leq v^*$.

\medskip
\textit{c)} Suppose that $\nu_{x}=\delta_{a_0}(x)$,
where $a_0 \in {\rm argmax} \{u(x, a)+\delta(v^*(f(x,a))\}$.
Then
\begin{align*}
P(v^*)(x)
&=
\ln
	\int_{a \in \Psi(x)}
		e^{u(x, a)+ \delta(v^*(f(x,a)))}
	\, d\nu_{x}(a)
\\[0.2cm]
&\leq
w^*(x)
\\[0.2cm]
&=
\ln  \big( e^{u(x, a_0)+\delta(v^*(f(x,a_0)))} \big)\, \delta_{a_0}(\Psi(x))
\\[0.2cm]
&=
\max_{a \in \Psi(x)} u(x, a)+\delta(v^*(f(x,a)))
\\[0.2cm]
&=
v^*(x),
\end{align*}
which implies that $w^*=v^*$.
\end{proof}

\subsection{Monotone Convergence Principles}

In this section we investigate the ordering  and the minimality of
the convergence of the iterations to the fixed points.
This topic is closely related to the theory of viscosity solutions
of Hamilton-Jacobi equations, where the subsolutions (supersolutions)
characterizes the original one.

\begin{lemma}[Monotonicity on $\delta$] \label{monotonicity bellman}
  Let $\delta_{1} \leq \delta_{2}$ be discount functions. If
  $v_1, v_2$ are solutions of Bellman's equation
  $v_j(x)= \max_{a \in \Psi(x)} u(x, a)+\delta_{j}(v_j(f(x,a)))$,
  $j=1,2$, then $v_{1} \leq v_{2}$. The same is
  true for the discounted transfer operator.
\end{lemma}
\begin{proof}
Since $\delta_{1} \leq \delta_{2}$ we have
$u(x, a)+\delta_{1}(v_1(f(x,a))) \leq u(x, a)+\delta_{2}(v_1(f(x,a)))$.
By taking the maximum over $\Psi(x)$
we obtain
\[
v_1(x) \leq \max_{a \in \Psi(x)} u(x, a)+\delta_{2}(v_1(f(x,a)))
=
B_{\delta_{2}}(v_1)(x).
\]
Iterating this inequality and using the fact that
$B_{\delta_{2}}^n(v_1)(x)\to v_2$, when $n \to \infty$,
we get $v_1 \leq v_2$.
\end{proof}

\begin{lemma}[Monotonicity on the operator] \label{monotonicity op bellman}
Let $v_1$ and $v_2$ be bounded functions, such that $v_{1} \leq v_{2}$.
Consider the Bellman operator
\[
B(v)(x)= \max_{a \in \Psi(x)} u(x, a)+\delta (v(f(x,a))).
\]
Then $B(v_{1}) \leq B(v_{2})$. In particular,
\begin{itemize}
\item[a)]
if $B(v) \leq v$ and $B(v^*)=v^*$ then $v^* \leq v$;

\item[b)]
if $B(v) \geq v$ and $B(v^*)=v^*$ then $v^* \geq v$.
\end{itemize}
The same is true for the discounted transfer operator.
\end{lemma}

\begin{proof} Since $\delta$ is an increasing function it follows that
  \begin{align*}
    B(v_1)(x)
    &=
    \max_{a \in \Psi(x)} u(x, a)+\delta (v_1(f(x,a)))
    \\
    & \leq
    \max_{a \in \Psi(x)} u(x, a)+\delta (v_2(f(x,a)))
    \\
    &=
    B(v_2)(x).
  \end{align*}

The statements \textit{a)} and \textit{b)} are proved in the same way.
Using the fact that $\delta$ is increasing we  obtain,
from the first part, $B(v)\leq  v$, $B^2(v)\leq B(v)\leq v$, etc.	
Recalling that the iterates $B^n(v) \to v^*$, when $n\to\infty$, for any initial $v$,
we obtain  $v^*\leq v$.
\end{proof}
\begin{remark}
The actual solution $v^*$ is minimal with respect
to the set of all subsolutions, that is,
$v^*\leq v$ for all $v$ satisfying $B(v)\leq  v$.
\end{remark}

\subsection{Regularity}
In this section we will establish the regularity of the fixed points
of the Koopman, Bellman and Discounted Transfer operators.
Such regularity properties will be proved under the following assumption.
\begin{assumption} \label{assump discount two variables}
	The contraction modulus  $\gamma$
	of the variable discount $\delta$ is also a variable discount function, and
	$\Psi(x)=\Psi(y), \; \forall x,y \in X$.
\end{assumption}
A particular case is when $\gamma=\delta$
(but they can be different, see Example~\ref{ex3discfunc})
and $\Psi(x)= A$, for all $x\in X$.

\begin{definition}[Joint sequential decision-making process]
\label{joint sequential decision making process}
\qquad\qquad\qquad \break
Let $S=\{X,A,\Psi,f,u,\delta\}$ be a sequential decision-making process
satisfying\break  Assumption~\ref{assump discount two variables}.
The joint sequential decision-making process
associated to $S$ is the decision-making process
$\hat{S}=\{X^2, A, \hat{\Psi}, \hat{f}, \hat{u}, \gamma\}$, where
	\begin{itemize}
		\item $\hat{\Psi}: X^2 \to A$ given by  $\hat{\Psi}(x,y)=\Psi(x) \subseteq A$
		is the set of all feasible actions for a agent $x$.
		
		\item $\hat{f}:X^2 \times A \to X^2$ is given by $\hat{f}(x,y,a)=(f(x,a),f(y,a))$.
		\item $\hat{u}:X^2 \times A \to \mathbb{R}$ is
		the immediate reward $\hat{u}(x,y,a)=|u(x,a)-u(y,a)|$;
	\end{itemize}
\end{definition}

\begin{definition} \label{variation aggregator}
Let $W: X\times A \times D  \to \mathbb{R}$ be an aggregator function
of the form $W(x,a,r):=u(x,a)+ \delta(r)$,
where $a \in \Psi(x)$. We define a new aggregator function
$\hat{W}: X^2\times A \times D  \to \mathbb{R}$ given by
\[
\hat{W}(x,y,a,r)
:=
\hat{u}(x,y,a) + \gamma(r)
=
|u(x,a)- u(y,a)|+ \gamma(r),
\]
where $\gamma$ is a contraction modulus for the variable discount $\delta$.
\end{definition}

\begin{lemma}\label{Domination Bellman solution}
Let  $v^* \in C_{b}(X, \mathbb{R})$ be the unique solution of the  Bellman equation
$v^*(x)= \max_{a \in \Psi(x)} u(x, a)+\delta(v^*(f(x,a))),$ provided
by Theorem~\ref{max op solution}. Let $w^* \in C_{b}(X, \mathbb{R})$
be the unique solution of the discounted transfer equation
\[
w^*(x)
=
\ln
	\int_{a \in \Psi(x)}
		e^{u(x, a)+\delta(w^*(f(x,a)))}
	\, d\nu_{x}(a),
\]
given by Theorem~\ref{disc transf op solution}. Then,
\[
|v^*(x) - v^*(y)|
\leq
\hat{V}^*(x,y)\text{ and  }|w^*(x) - w^*(y)|
\leq
\hat{V}^*(x,y), \; \forall x, y \in X,
\]
where $\hat{V}^*$ is the unique fixed point of the Bellman operator
\[
\hat{B}(\hat{V})(x,y)
:=
\sup_{a \in \hat{\Psi}(x,y)}
\hat{u}(x,y,a)+\gamma(\hat{V}(\hat{f}(x,y,a))).
\]
\end{lemma}

\begin{proof}
From the definition and triangular inequality, we get
\begin{align*}
v^*(x)
&=
\max_{a \in \Psi(x)} u(x, a)+\delta(v^*(f(x,a))
\\
&\leq
\max_{a \in \Psi(x)}
 \left\{ u(x, a)- u(y, a)+ u(y, a)+ \delta(v^*(f(x,a)) \right.
\\
&\qquad\qquad\
\left. - \delta(v^*(f(y,a)) + \delta(v^*(f(x,a))\right\}
\\[0.2cm]
&\leq
\max_{a \in \Psi(y)} \left\{ u(y, a)+\delta(v^*(f(y,a)))\right\}
+ |u(x, a)- u(y, a)|
\\
&\qquad\qquad\
+ |\delta(v^*(f(x,a))- \delta(v^*(f(y,a)) |
\\[0.2cm]
&\leq
v^*(y) + \hat{u}(x,y,a) + \gamma(|v^*(f(x,a) - v^*(f(y,a)|).
\end{align*}
By a similar reasoning, replacing $x$ by $y$, we obtain
\[
|v^*(x) - v^*(y)| \leq \hat{u}(x,y,a) + \gamma(|v^*(f(x,a)) - v^*(f(y,a))|).
\]
Analogously,
$
|w^*(x) - w^*(y)|
\leq
\hat{u}(x,y,a) + \gamma(|w^*(f(x,a) - w^*(f(y,a))|),
$
since
\begin{align*}
|w^*(x) - w^*(y)|
&=
\ln{ \frac{\int_{a \in \Psi(x)} e^{u(x, a)+\delta(w^*(f(x,a)))}\, d\nu_{x}(a)}
	      {\int_{a \in \Psi(y)} e^{u(y, a)+\delta(w^*(f(y,a)))}\, d\nu_{y}(a)}
}
\\
&\leq
\sup_{a \in \Psi(x,y)}
|u(x, a)- u(y, a)| + |\delta(w^*(f(x,a)))- \delta(w^*(f(y,a)))|.
\end{align*}
In both cases, where $\zeta(x,y)=|v^*(x) - v^*(y)|$ or $\zeta(x,y)=|w^*(x) - w^*(y)|$
\footnote{the remaining of the argument works
for both choices, because it depends only on the monotonicity properties,
so all this formalism works equally to both families of fixed points $v^*$ and $w^*$.
Since $|v^*(x) - v^*(y)|
\leq
\hat{V}^*(x,y)\text{ and  }|w^*(x) - w^*(y)|
\leq
\hat{V}^*(x,y), \; \forall x, y \in X$.
Note that $u$ is the same in both cases.
}
we obtain $\hat{B}(\zeta)  \leq \zeta$, where
\[
\hat{B}(\hat{V})(x,y)
=
\sup_{a \in \hat{\Psi}(x,y)}  \hat{u}(x,y,a)+\gamma(\hat{V}(\hat{f}(x,y,a)).
\]
From Lemma~\ref{monotonicity op bellman} follows that $\zeta \leq \hat{V}^*$,
where $\hat{V}^*$ is the unique solution of the Bellman operator $\hat{B}$.
By Assumption~\ref{assump discount two variables} and
Theorem \ref{max op solution} there exists a unique $\hat{U}$
solving the Koopman equation
\[
\hat{U}(\hat{h}_{(x,y)} \bar{a}))
=
\hat{K}(\hat{U})(\hat{h}_{(x,y)} \bar{a})))
:=
\hat{u}(x,y,a)+\gamma(\hat{U}(\hat{\sigma}(\hat{h}_{(x,y)} \bar{a}))),
\]
such that
\[
\hat{V}^*(x,y)
=
\hat{U}(\hat{h}_{(x,y)} \bar{a})
=
{\sum_{i}}^{\star}
\hat{u}(x_{i},y_{i},a_{i}^*),
\]
for some optimal plan $a^* \in \Pi(x,y)$\footnote{
is the set of feasible action sequences for the joint
sequential decision making process $\hat{S}=\{X^2, A, \hat{\Psi}, \hat{f}, \hat{u}, \gamma\}$
}.
\end{proof}

\begin{lemma}\label{pseudometric}
Let $\hat{V}^*$ be the unique fixed point of the Bellman operator
\[
\hat{B}(\hat{V})(x,y)
=
\sup_{a \in \hat{\Psi}(x,y)}
\hat{u}(x,y,a)+\gamma(\hat{V}(\hat{f}(x,y,a)).
\]
Then
\begin{itemize}
\item[a)]
$\hat{V}^*(x,y) \geq 0$ and  $\hat{V}^*(x,x)=0$;

\item[b)]
$\hat{V}^*(x,y)= \hat{V}^*(y,x)$;
\end{itemize}
That is, $V^*: X^2 \to \mathbb{R}$ is a symmetric and nonnegative function.
In particular, from optimality of the solutions of Bellman's equation we have
\[
\hat{V}^*(x,y)
:=
\sup_{\hat{h}_{(x,y)} \bar{a} \in
	\hat{\Omega}}\hat{U}(\hat{h}_{(x,y)} \bar{a})
=
{\sum_{i}}^{\star}
\hat{u}(x_{i},y_{i},a_{i}^*),
\]
for some optimal plan $a^* \in \Pi(x,y)$.
\end{lemma}
\begin{proof}
\textit{a)} We recall that
$
\hat{V}^*(x,y)
:=
\hat{U}(\hat{h}_{(x,y)} \bar{a})
=
\sum_{i}^{\star}
\hat{u}(x_{i},y_{i},a_{i}^*),
$
for some optimal plan $a^* \in \Pi(x,y)$. Since $\gamma$ is assumed to be
a discounting function therefore increasing, and $\hat{u}\geq 0$, we have
immediately
\[
\hat{V}^*(x,y)
=
\hat{u}(x_{0},y_{0},a_{0}^*)
+
\gamma\Big( \sum_{i}^{\star} \hat{u}(x_{i+1},y_{i+1},a_{i+1}^*)  \Big)
\geq
0.
\]
By definition $\hat{u}(x,x,a)=|u(x,a)-u(x,a)|=0$ so $\hat{V}^*(x,x)=0$.
\\

\textit{b)} By definition $\hat{u}(x,y)= |u(x,a)-u(y,a)|=|u(y,a)-u(x,a)|=\hat{u}(y,x)$.
Let us define $U'$ as the unique solution of the Koopman equation
\[
U'(\hat{h}_{(x,y)} \bar{a}))
=
\hat{u}(x,y,a)+\gamma(U'(\hat{\sigma}(\hat{h}_{(x,y)} \bar{a}))),
\]
and
$U''(\hat{h}_{(x,y)} \bar{a}))=U'(\hat{h}_{(y,x)} \bar{a}))$.
Then it satisfies
\begin{align*}
U''(\hat{h}_{(x,y)} \bar{a}))
&=
\hat{u}(y,x,a)+\gamma(U'(\hat{\sigma}(\hat{h}_{(y,x)} \bar{a})))
\\
&=
\hat{u}(x,y,a)+\gamma(U''(\hat{\sigma}(\hat{h}_{(x,y)} \bar{a}))).
\end{align*}
By the uniqueness we obtain $U''=U$ thus,
$U'(\hat{h}_{(x,y)} \bar{a}))= U'(\hat{h}_{(y,x)} \bar{a}))$,
which is equivalent to $\hat{V}^*(x,y)= \hat{V}^*(y,x)$.
\end{proof}

\begin{definition} \label{nondegenerated definition}
We say that $u: X \times A \to \mathbb{R}$ is non-degenerated if
for any $x \neq y$ in $X$ there exists $n \in \mathbb{N}$ and
$\bar{a} \in \Pi(x,y)$ such that
\[
u(x_n, a_n) \neq u(y_n, a_n)
\]
where $(x_i, a_i)_{i \in \mathbb{N}},  (y_i, a_i)_{i \in \mathbb{N}} \in \Omega$.
\end{definition}

Of course, if for any fixed $a \in A$, the function $ x \longmapsto u(x, a)$
is strictly increasing (or decreasing) then $u$ is non-degenerated.
Therefore there is at least one very natural
sufficient condition to non-degeneration.

\begin{theorem} \label{twist imply metric}
  If $u$ is non-degenerated then $\hat{V}^*(x,y)$ is separating,
  that is, if $\hat{V}^*(x,y)=0$ then $x=y$.
\end{theorem}
\begin{proof}
Suppose that $x\neq y$, but $\hat{V}^*(x,y)=0$.
Then, by optimality we have
\[
0
=
\hat{V}^*(x,y):=\sup_{\hat{h}_{(x,y)}
\bar{a} \in \hat{\Omega}}\hat{U}(\hat{h}_{(x,y)} \bar{a})
\]
so $\hat{U}(\hat{h}_{(x,y)} \bar{a})=0$, that is,
$\sum_{i}^{\star}\hat{u}(x_{i},y_{i},a_{i}^*)=0,$
where $\bar{a}=(a_0, a_1, \ldots)$.
Since $\gamma$ is an increasing function, with $\gamma(0)=0$,
and $\hat{u}\geq 0$, we have $\hat{u}(x_{i},y_{i},a_{i})=0$
for all $i \geq 0$, thus contradicting the
non-degeneration property of $u$.
\end{proof}

\subsection{Discounted Limits}\label{section-discounted-limits}
In this section we consider the limits of fixed points of a variable
discount decision-making process defined by a continuous and
bounded immediate reward  $u:X \times A \to \mathbb{R}$ and
a sequence $(\delta_n)_{n\geq0}$ of
discounts $\delta_{n}:[0,+\infty) \to \mathbb{R}$,
satisfying $\delta_{n}(t) \to I(t)=t$, when $n\to\infty$,
in the pointwise topology.
For instance, $\delta_n(t)=t(n-1)/n + (1/n)\ln(1+t)$
is a nonlinear, idempotent ($\gamma_n(t)=\delta_n(t)$) and
subadditive discount function. It is easy to see that
$\delta_n(t)\to t$, when $n\to\infty$, for all $t\geq 0$.

Under these assumptions we want to study the sequences
\[
v_n^*(x)
:=
\max_{a \in \Psi(x)} u(x, a)+\delta_{n}(v_{n}^*(f(x,a)))
\]
and
\[
w_{n}^*(x)
:=
\ln \int_{a \in \Psi(x)} e^{u(x, a)+\delta_{n}(w_{n}^*(f(x,a)))} d\nu_{x}(a),
\]
and investigate whether their normalizations
\[
(v_n^*(x)-\sup_{x}v_n^*(x))_{n\geq 0}
\quad\text{and}\quad
(w_n^*(x)-\sup_{x}w_n^*(x))_{n\geq 0}
\]
have some cluster points
$v_{\infty}, w_{\infty} \in C_{b}(X, \mathbb{R})$, solving the equations
\[
v_{\infty}(x)
:=
\max_{a \in \Psi(x)}  (u(x, a) - \alpha) +v_{\infty}(f(x,a))
\]
and
\[
e^{k} e^{w_{\infty}(x)}
=
\int_{a \in \Psi(x)} e^{u(x, a)} e^{w_{\infty}(f(x,a))} d\nu_{x}(a),
\]
for some $\alpha, k \in \mathbb{R}$.
The first one is the subaction equation in ergodic optimization and the
second is the eigenfunction equation for the Ruelle operator.

\begin{assumption} \label{assump not decreas}
  We assume that $\delta_{n}(t) \leq t$, for all $n \geq 0$.
\end{assumption}

Since $\delta_{n}(0)=0$, we can construct examples satisfying the above
condition by requiring that
$q_n(t)=t - \delta_{n}(t)$, for all $n \geq 0$, is not decreasing.

\begin{lemma} \label{constant bellman eq}
Under Assumption~\ref{assump not decreas} we have
\[
0
\leq
M_n - \delta_{n}(M_n)
\leq
\| u \|_{\infty},
\]
where
\[
M_n= \sup_{x\in X} v_n(x)
\qquad\text{and}\qquad
v_n(x)= \max_{a \in \Psi(x)} u(x, a)+\delta_{n}(v_{n}(f(x,a))
\]
or
\[
M_n= \sup_{x\in X} w_n(x)
\qquad\text{and}\qquad
w_{n}(x)=\ln \int_{a \in \Psi(x)} e^{u(x, a)+\delta_{n}(w_{n}(f(x,a)))} d\nu_{x}(a).
\]
\end{lemma}
\begin{proof}
\textbf{Case $M_n= \sup_{x\in X} v_n(x)$:} by using Bellman's equation we obtain
\begin{align*}
u(x, a)+\delta_{n}(v_{n}(f(x,a))
&\leq
\| u \|_{\infty} + \delta_{n}(M_n)
\\
v_n(x)
&\leq
\| u \|_{\infty} + \delta_{n}(M_n)
\\
M_n
&\leq
\| u \|_{\infty} + \delta_{n}(M_n)
\\
M_n - \delta_{n}(M_n)
&\leq
\| u \|_{\infty},
\end{align*}
By hypothesis we have $0 \leq M_n - \delta_{n}(M_n)$ and so
follows from previous inequality that
$0 \leq M_n - \delta_{n}(M_n)  \leq \| u \|_{\infty}$.
\\

\noindent\textbf{Case $M_n= \sup_{x\in X} w_n(x)$:}
by using discounted transfer operator  equation we obtain
\begin{align*}
u(x, a)+\delta_{n}(w_{n}(f(x,a))
&\leq
\| u \|_{\infty} + \delta_{n}(M_n)
\\
e^{u(x, a)+\delta_{n}(w_{n}(f(x,a))}
&\leq
e^{\| u \|_{\infty} + \delta_{n}(M_n)}
\\
w_n(x)
&=
\ln \int_{a \in \Psi(x)} e^{u(x, a)+ \delta_{n}(w_{n}(f(x,a))}d\nu_{x}(a)
\\
&\leq
\ln \int_{a \in \Psi(x)} e^{\| u \|_{\infty} + \delta_{n}(M_n)}d\nu_{x}(a)
\\
&=
\| u \|_{\infty} + \delta_{n}(M_n)
\\
M_n &\leq \| u \|_{\infty} + \delta_{n}(M_n)
\\
M_n - \delta_{n}(M_n)  &\leq \| u \|_{\infty}.
\end{align*}

By using the above inequality and proceeding as in the
previous case we get
$0 \leq M_n - \delta_{n}(M_n)  \leq \| u \|_{\infty}$.
\end{proof}

We point out that $\delta(t)= \beta t$, $\beta \in (0,1)$,
$\delta(t)= \ln (1+t)$, and
$\delta(t)= \sum_{i=0}^{\infty} (\beta_i t +\alpha_i) \chi_{[i,i+1)}$
with $\beta_i \searrow 0$ satisfies the condition that
$q(t)=t - \delta(t)$ is not a decreasing function.

\begin{definition}\label{gamma continuity}
Given a  discount $\delta$, the return function $u$ is called
\begin{itemize}
\item[a)]
$\delta$-bounded if $\sum_{i}^{\star} \hat{u}(x_{i},y_{i},a_{i}) \leq C_{\delta}$;

\item[b)] $\delta$-dominated if
\[
\lim_{\theta \to 0}\
\sup_{d_X(x,y) \leq \theta}\
\sup_{\bar{a}\in\Pi(x,y)}
{\sum_{i}}^{\star} \hat{u}(x_{i},y_{i},a_{i})
=
0.
\]
\end{itemize}
Given a family of discount functions $(\delta_n)_{n\geq0}$ we say that $u$ is
\begin{itemize}
\item[a)]
uniformly $\delta$-bounded if $u$ is $\delta_{n}$-bounded for
all $n$ and $\sup_{n} C_{\delta_{n}}=C<+\infty$.

\item[b)]
uniformly $\delta$-dominated if $u$ is $\delta_{n}$-dominated
for all $n$ and
\[
\lim_{\theta \to 0}\
\sup_{n\in\mathbb{N}}\
\sup_{d_X(x,y) \leq \theta}\
\sup_{\bar{a}\in\Pi(x,y)}
{\sum_{i}}^{(\star,\gamma_n)} \hat{u}(x_{i},y_{i},a_{i})
=0
\]
where ${\sum_{i}}^{(\star,\gamma_n)}$ is
$\sum_{i}^{\star}$ with the discount variable
$\delta_n$.
\end{itemize}
\end{definition}

The next theorem shows that the class of uniformly $\delta$-dominated
contains the class of Lipschitz or $\alpha$-H\"older potentials,
when the dynamics of the decision process is uniformly contractive.

\begin{theorem}\label{regularity implies domination}
Suppose that the dynamics $f$ is contractive, that is,
\[
\sup_{a \in A} d_{X}(f(x,a),f(y,a)) \leq \lambda d_{X}(x,y).
\]
If $u(\cdot, a)$ is $C$-Lipschitz (or $\alpha$-H\"older) then $u$
is uniformly $\delta$-dominated. In addition, if $\mathrm{diam}(X) < \infty$
then $u$ is uniformly $\delta$-bounded. In particular,
if $v_n$ and $w_n$ are respectively the solutions of Bellman's
equation and the transfer discounted operator equation,
they are uniformly $C(1-\lambda)^{-1}$-Lipschitz (or $\alpha$-H\"older,
with ${\rm Hol}_{\alpha}(v_n)={\rm Hol}_{\alpha}(w_n)={\rm Hol}_{\alpha}(u) (1-\lambda^{\alpha})^{-1}$).
\end{theorem}
\begin{proof}
\textbf{Case 1: }$u(\cdot, a)$ is $C$-Lipschitz, that is, $|u(x,a)-u(y,a)|\leq C d_{X}(x,y)$.
In this case for any pair $x,y\in X$ satisfying $d_X(x,y) \leq \theta$, we have
the following estimate
$
\hat{u}(x_{i},y_{i},a_{i})
=
|u(x_{i},a_{i})-u(y_{i},a_{i})|
\leq
C  d_{X}(x_i,y_i)
\leq
C d_{X}(x,y)\lambda^i
\leq
C\theta \lambda^i
$,
which immediately implies
\begin{align*}
{\sum_{i}}^{(\star,\gamma_n)} \hat{u}(x_{i},y_{i},a_{i})
\leq
\sum_{i=0}^{\infty} C \lambda^i \theta
=
\frac{C\theta}{1-\lambda}
\end{align*}
because $\gamma_n(x) < x$, for all $n\geq 0$. Thus,
\[
\lim_{\theta \to 0}\
\sup_{n\in\mathbb{N}}\
\sup_{d_X(x,y) \leq \theta}\
\sup_{\bar{a}\in\Pi(x,y)}
{\sum_{i}}^{(\star,\gamma_n)} \hat{u}(x_{i},y_{i},a_{i})
\leq
\lim_{\theta \to 0}
\frac{C\theta}{1-\lambda}
=
0.
\]

\noindent\textbf{Case 2: }$u(\cdot, a)$ is $\alpha$-H\"older, that is,
$|u(x,a)-u(y,a)|\leq {\rm Hol}_{\alpha}(u) d_{X}(x,y)^\alpha$,
for $0<\alpha<1$. A similar reasoning shows that
$
\hat{u}(x_{i},y_{i},a_{i})
\leq {\rm Hol}_{\alpha}(u)
d_{X}(x_i,y_i)^\alpha
\leq
{\rm Hol}_{\alpha}(u) \lambda^{\alpha i} d_{X}(x,y)^{\alpha}
\leq
{\rm Hol}_{\alpha}(u) (\lambda^{\alpha})^i  \theta^{\alpha}
$
and
\begin{align*}
{\sum_{i}}^{(\star,\gamma_n)} \hat{u}(x_{i},y_{i},a_{i})
\leq
\sum_{i=0}^{\infty} {\rm Hol}_{\alpha}(u) (\lambda^{\alpha})^i \theta^{\alpha}
=
\frac{{\rm Hol}_{\alpha}(u) }{1-\lambda^{\alpha}} \theta ^{\alpha}
\xrightarrow{\ \theta\to 0\ } 0.
\end{align*}
Thus proving that $u$ is uniformly $\delta$-dominated.
The uniform $\delta$-boundedness is trivial from the above computations
as long as  $\mathrm{diam}(X)<\infty$.
\\

To prove the last claim (assuming Lipschitz condition),
we use Lemma~\ref{Domination Bellman solution} and the inequalities
\[
|v_n(x) - v_n(y)|
\leq
\hat{V}^*(x,y)\text{ and  }|w_n(x) - w_n(y)|
\leq
\hat{V}^*(x,y), \; \forall x, y \in X.
\]
By similar computations, replacing $\theta$ by $d_{X}(x,y)$, we obtain
\[
|v_n(x) - v_n(y)|
\leq
\hat{V}^*(x,y)
=
{\sum_{i}}^{(\star,\gamma_n)} \hat{u}(x_{i},y_{i},a_{i}^{*})
\leq
\frac{C }{1-\lambda} d_{X}(x,y).
\]
Analogously for the $\alpha$-H\"older case.
\end{proof}

\begin{lemma} \label{boundness bellman}
Let the contraction modulus  $\gamma_{n}$ of the variable discount
$\delta_{n}$ be also a variable discount function,
$\Psi(x)=\Psi(y), \; \forall x,y \in X$ and $u$ uniformly
$\delta$-dominated.
Then $\bar{v}_n= v_n(x) - M_n$, where
$v_n(x)= \max_{a \in \Psi(x)} \{u(x, a)+\delta_{n}(v_{n}(f(x,a))\}$
is uniformly bounded, that is,
\[
-2C \leq \bar{v}_n \leq 0.
\]
The same is true for $\bar{w}_n= w_n(x) - M_n$, where
\[w_{n}(x)=\ln \int_{a \in \Psi(x)} e^{u(x, a)
+\delta_{n}(w_{n}(f(x,a))} d\nu_{x}(a).
\]
\end{lemma}
\begin{proof}
We give the argument for $v_n$. The proof for $w_n$ is similar.
We already know that
$
| v_n(x) -v_n(y)|
\leq
{\sum_{i}}^{(\star,\gamma_n)} \hat{u}(x_{i},y_{i},a_{i}^{*})
\leq
C,
$
for some optimal plan $a^* \in \Pi(x,y)$,
uniformly in $n$.

Obviously $\bar{v}_n (x)=v_n(x) -M_n \leq 0$.
On the other hand, we get from the hypothesis
$-C\leq v_n(x) - v_n(y) $ and subtracting $M_n$ we obtain
$ v_n(y)-M_n -C \leq v_n(x) - M_n$ or,
$ m_n-M_n -C \leq \bar{v}_n (x)$, where $m_n=\min v_n$.
Since $| v_n(x) -v_n(y)| \leq C,$
it follows that $M_n -m_n \leq C$ and so $m_n -M_n \geq -C$.
Thus, $ -C -C \leq \bar{v}_n (x)$, which implies $ -2C \leq \bar{v}_n (x)$.
\end{proof}

Now we present a sufficient condition for both families of fixed points
to be equicontinuous, under normalization.

\begin{lemma}\label{equicontinuous bellman}
Under the hypothesis of Lemma~\ref{boundness bellman},
if $u$ is uniformly $\delta$-dominated with respect to $(\delta_n)_{n\geq0}$,
then the families $\bar{v}_n (x)=v_n(x) -M_n$ and
$\bar{w}_n (x)=w_n(x) -M_n$ are equicontinuous.
\end{lemma}
\begin{proof}
From Lemma~\ref{Domination Bellman solution} we know that
\begin{align*}
\hat{V}^*_n (x,y)
=
\sup_{\bar{a}\in\Pi(x,y)}
{\sum_{i}}^{(\star,\gamma_n)} \hat{u}(x_{i},y_{i},a_{i}^{*})
\geq
|v_n(x) - v_n(y)|
=
|\bar{v}_n(x) - \bar{v}_n(y)|,
\end{align*}
that is, the modulus of uniform continuity
$\omega(\bar{v}_n, \theta)$ of $\bar{v}_n$ satisfies
\begin{align*}
\omega(\bar{v}_n, \theta)
=
\sup_{d_X(x,y) \leq \theta}|\bar{v}_n(x) - \bar{v}_n(y)|
\leq
\sup_{d_X(x,y) \leq \theta}\
\sup_{\bar{a}\in\Pi(x,y)}
{\sum_{i}}^{(\star,\gamma_n)} \hat{u}(x_{i},y_{i},a_{i}^{*}).
\end{align*}
Thus, for any $\varepsilon>0$ there exists $\theta>0$ such that,
$|\bar{v}_n(x) - \bar{v}_n(y)|<\varepsilon$
provided that $d_X(x,y) \leq \theta$ and it is independent of $n$.
\end{proof}

\begin{assumption} \label{assump not assymp}
For any fixed $\alpha >0$ we have
\[
\lim_{n\to \infty}\delta_{n}(t +\alpha) - \delta_{n}(t)=\alpha,
\]
uniformly for $t>0$.
\end{assumption}

Examples where the above assumption is satisfied are given by
\[
\delta_n(t)=\frac{n-1}{n} t + \frac{1}{n}\ln(1+t)
\quad
\text{and}
\quad
\delta_n(t)=\frac{n-1}{n} t + \frac{1}{n}(-1+ \sqrt{1+t}).
\]

\begin{theorem} \label{bellman subaction}
If the assumptions of Lemmas \ref{boundness bellman} and \ref{equicontinuous bellman}, and Assumption~\ref{assump not assymp} are assumed to hold. Then
there exists a value $\bar{u} \in [0 , \| u \|_{\infty} ]$
and a function $h$ such that
$h(x)= \max_{a \in \Psi(x)} u(x, a)- \bar{u} +  h(f(x,a)).$
\end{theorem}
\begin{proof}
We consider the sequence of functions $\bar{v}_n (x)=v_n(x) -M_n$
and the discounted limit $\delta_{n} \to Id_{D}$.
Since each $v_n$ satisfy Bellman's equation we have
\begin{align*}
  v_n(x)
  &=
  \max_{a \in \Psi(x)} u(x, a)+\delta_{n}(v_{n}(f(x,a))
  \\
  v_n(x)-M_n
  &=
  \max_{a \in \Psi(x)} u(x, a)+\delta_{n}(v_{n}(f(x,a)) -M_n
  \\
  \bar{v}_n (x)
  &=
  \max_{a \in \Psi(x)} u(x, a)+ \delta_{n}(v_{n}(f(x,a))- \delta_{n}(M_n) + \delta_{n}(M_n) - M_n
  \\
  \bar{v}_n (x) &=
  \max_{a \in \Psi(x)} u(x, a) -
  \left(M_n-\delta_{n}(M_n)\right)+ \delta_{n}(v_{n}(f(x,a))- \delta_{n}(M_n).
\end{align*}
From Lemma~\ref{constant bellman eq} we know that
$0 \leq M_n - \delta_{n}(M_n)  \leq \| u \|_{\infty}$ so,
possibly choosing a subsequence we can find $\bar{u} \in [0 , \| u \|_{\infty} ]$
such that  $M_n - \delta_{n}(M_n) \to \bar{u}$ when $n \to \infty$.
From Lemma~\ref{boundness bellman} and Lemma~\ref{equicontinuous bellman}
the sequence $\bar{v}_n$ is uniformly bounded and equicontinuous.
From Arzel\`a-Ascoli's theorem we obtain a subsequence
(that we still calling $\bar{v}_n$ to avoid extra indexes)
that converges to a continuous function $h$ satisfying
$h(x)= \max_{a \in \Psi(x)} u(x, a)- \bar{u} +  h(f(x,a)),$
if $\delta_{n}(v_{n}(f(x,a)))- \delta_{n}(M_n) \to h(f(x,a))$
when $ v_n(x)-M_n \to h(x)$. To prove that we recall that,
from the definition of variable discount function $\delta_n$,
it is increasing so we have
\[
\delta_{n}(M_n)-\delta_{n}(v_{n}(x))
\leq
\gamma_n (M_n- v_n(x)).
\]
Since
$M_n - v_n(x) \to -h(x)\geq 0$,
we can conclude that for $n$ big enough that
$-h(x)- \varepsilon \leq  M_n - v_n(x) \leq  -h(x) +\varepsilon$,
or equivalently
\[
v_n(x)-h(x)- \varepsilon \leq  M_n \leq  v_n(x)-h(x) +\varepsilon.
\]
Using the fact that $\delta_n$ is increasing we obtain
\[
\delta_{n}(v_n(x)-h(x)- \varepsilon)
\leq
\delta_{n}(M_n)
\leq
\delta_{n}(v_n(x)-h(x) +\varepsilon).
\]
By adding $-\delta_{n}(v_n(x))$, we obtain
\begin{align*}
\delta_{n}(v_n(x)-h(x)- \varepsilon)-\delta_{n}(v_n(x)
&\leq
\delta_{n}(M_n) -\delta_{n}(v_n(x))
\\
&\leq
\delta_{n}(v_n(x)-h(x) +\varepsilon)-\delta_{n}(v_n(x).
\end{align*}
Now, from Assumption~\ref{assump not assymp},
it follows that
\[
\lim_{n\to \infty} \delta_{n}(M_n) -\delta_{n}(v_n(x))=-h(x).
\qedhere
\]
\end{proof}

\begin{remark}\label{other discounts}
We can consider other families of $\delta_{n}$'s assuming
the same hypothesis except for Assumption~\ref{assump not assymp}.
For example, the family $\delta_n(t)=(-1+ \sqrt{1+t})$ satisfies:
for any fixed $\alpha >0$, we have
$\lim_{n\to \infty}\delta_{n}(t +\alpha) - \delta_{n}(t)=0$,
uniformly on $t>0$.
In this case, the discount limit will produce an equation
$h(x)= \max_{a \in \Psi(x)} u(x, a)- \bar{u}$,
having a very different meaning.
\end{remark}

\begin{remark}\label{equivalence subaction and HJB equation}
In ergodic optimization this function $h$ is called a calibrated subaction
of $u$ with respect to the dynamics $f$.
In the theory of viscosity solutions of the Hamilton-Jacobi-Bellman equations,
the equation $h(x)= \max_{a \in \Psi(x)} u(x, a)- \bar{u} +  h(f(x,a))$
can be rewritten as
\[
\bar{u}
=
\max_{a \in \Psi(x)} u(x, a) +  h(f(x,a)) -h(x)
=
\max_{a \in \Psi(x)} d_{x}h(a) + u(x, a)= H(x, d_{x}h),
\]
where the discrete differential is
$d_{x}h(a)= h(f(x,a)) -h(x)$
and the Hamiltonian $H$ is the Legendre transform of $u$.
\end{remark}

Recall that the set of holonomic probability measures is defined by
\[
\mathcal{H}
:=
\left\{
\mu \in \mathscr{P}(\Omega)
\; | \;
\int_{\Omega}d_{x}g (a)\, d\mu(x,a)=0, \; \forall g \in C(A,\mathbb{R})
\right\}.
\]

\begin{theorem}\label{ergodic optim for IFS}
Assume that the hypothesis of Theorem~\ref{bellman subaction}
are satisfied and put
\[
\check{u}
=
\sup_{\mu \in \mathcal{H}} \int_{\Omega} u(x,a)\, d\mu(x,a).
\]
If $\Omega$ is compact then $\check{u}=\bar{u}$, in particular,
the number given by Theorem~\ref{bellman subaction} is unique.
\end{theorem}
\begin{proof}
From Remark~\ref{equivalence subaction and HJB equation} we know that
$\bar{u}= \max_{a \in \Psi(x)} d_{x}h(a) + u(x, a) \geq d_{x}h(a) + u(x, a)$
and integrating with respect to $\mu \in \mathcal{H}$ we obtain
\[
\bar{u} \geq \int_{\Omega}d_{x}h (a) d\mu(x,a) +
\int_{\Omega}u(x, a) d\mu(x,a)
=
\int_{\Omega}u(x, a) d\mu(x,a),
\]
thus $\bar{u} \geq \check{u}$.

To show the equality we will built a holonomic maximizing probability.
Inductively, we choose $a_0 \in \Psi(x)$ such that
$\bar{u}=  d_{x}h(a_0) + u(x, a_0)$, $a_1 \in \Psi(x_1)$ such that
$\bar{u}=  d_{x_1}h(a_1) + u(x_1, a_1)$, and so on.
Notice that $x_0=x$ and $x_{n+1}=f(x_{n},a_{n})$,
for all $n \geq 0$.
Define a probability measure $\mu_{k}$ by
\[
\mu_{k}(g) := \frac{1}{k}\sum_{i=0}^{k-1} g(x_i, a_i)
\]
then, adding the above equations we get
$k \bar{u} = \sum_{i=0}^{k-1} d_{x_i}h(a_i) + u(x_i, a_i)$ or equivalently
\begin{align*}
\bar{u}
&=
\frac{1}{k}\sum_{i=0}^{k-1} d_{x_i}h(a_i) +
\frac{1}{k} \sum_{i=0}^{k-1} u(x_i, a_i)
\\[0.3cm]
&=
\frac{h(x_{k-1})- h(x_0)}{k} + \int_{\Omega} u(x,a)\, d\mu_{k}(x,a).
\end{align*}
Since $h$ is bounded and $\Omega$ is compact,
up to subsequence, we can assume that
$\mu_{k}\rightharpoonup \mu$.
A straightforward calculation shows that  $\mu \in \mathcal{H}$
and
\[
\bar{u} = \int_{\Omega}u(x, a)\, d\mu(x,a).
\qedhere
\]
\end{proof}

\begin{theorem} \label{ruelle theorem discounted}
If the assumptions of Lemmas \ref{boundness bellman} and \ref{equicontinuous bellman}, and Assumption~\ref{assump not assymp} are assumed to hold. Then
there exists a value $k \in [0 , \| u \|_{\infty} ]$ and a function $h$
given by  $h(x)=\ln \int_{a \in \Psi(x)} e^{u(x, a)+h(f(x,a)) - k} d\nu_{x}(a),$
such that $\rho:=e^{k}$  and $\varphi:= e^{h(x)}$ are a
positive eigenvalue and a positive and continuous eigenfunction, respectively,
for Ruelle operator, i.e.,
\[
e^{k} e^{h(x)}=\int_{a \in \Psi(x)} e^{u(x, a)} e^{h(f(x,a))}
\, d\nu_{x}(a).
\]
\end{theorem}
\begin{proof}Consider
\[
M_n= \sup_{x\in X} w_n(x)
\quad\text{and}\quad
w_{n}(x)=\ln \int_{a \in \Psi(x)} e^{u(x, a)+\delta_{n}(w_{n}(f(x,a)))}
\, d\nu_{x}(a).
\]
Take the sequence of functions $\bar{w}_n (x)=w_n(x) -M_n$ and analyze
the discounted limit $\delta_{n} \to Id_{D}$.
Since each $w_n$ satisfies
the discounted transfer operator equation we have
\[
w_n(x)=\ln \int_{a \in \Psi(x)} e^{u(x, a)+\delta_{n}(w_{n}(f(x,a)))} d\nu_{x}(a).
\]
From this equality follows that
\begin{align*}
  w_n(x)-M_n
  &=
  \ln \int_{a \in \Psi(x)} e^{u(x, a)+\delta_{n}(w_{n}(f(x,a)))-M_n}
  \, d\nu_{x}(a)
  \\
  &=
  \ln \int_{a \in \Psi(x)} e^{u(x, a)+\delta_{n}(w_{n}(f(x,a)))-\delta_{n}(M_n)
  + \delta_{n}(M_n)-M_n}
  \, d\nu_{x}(a).
\end{align*}
By taking exponential on both sides we get
\[
  e^{w_n(x)-M_n}
  =
  \int_{a \in \Psi(x)} e^{u(x, a)+\delta_{n}(w_{n}(f(x,a)))-\delta_{n}(M_n) -
  (M_n-\delta_{n}(M_n))} d\nu_{x}(a)
\]
which in turn implies
\[
e^{M_n-\delta_{n}(M_n)} e^{w_n(x)-M_n}
=
\int_{a \in \Psi(x)} e^{u(x, a)}
e^{\delta_{n}(w_{n}(f(x,a)))-\delta_{n}(M_n) } d\nu_{x}(a).
\]

From Lemma~\ref{constant bellman eq} we know that
$0 \leq M_n - \delta_{n}(M_n)  \leq \| u \|_{\infty}$ so, possibly assuming
a subsequence we can find $k \in [0 , \| u \|_{\infty} ]$ such that
$M_n - \delta_{n}(M_n) \to k$ when $n \to \infty$. From Lemma~\ref{boundness bellman}
and Lemma~\ref{equicontinuous bellman} the sequence $\bar{w}_n$ is uniformly bounded
and equicontinuous. From Arzel\`a-Ascoli's theorem we obtain
a subsequence (that we still calling $\bar{w}_n$ to avoid extra indexes)
that converges to a continuous function $h$ satisfying
$e^{k} e^{h(x)}=\int_{a \in \Psi(x)} e^{u(x, a)} e^{h(f(x,a))} d\nu_{x}(a)$
if $\delta_{n}(w_{n}(f(x,a)))- \delta_{n}(M_n) \to h(f(x,a))$ when $ w_n(x)-M_n \to h(x)$.
To prove this we recall that, from the definition of variable
discount function $\delta_n$, it is increasing so we have
\[
\delta_{n}(M_n)-\delta_{n}(w_{n}(x))
\leq
\gamma_n (M_n- w_n(x))
\]
Since
$M_n - w_n(x) \to -h(x)\geq 0$, we can conclude that for $n$ big enough we
have
$-h(x)- \varepsilon \leq  M_n - w_n(x) \leq  -h(x) +\varepsilon$,
or equivalently
\[
w_n(x)-h(x)- \varepsilon \leq  M_n \leq  w_n(x)-h(x) +\varepsilon.
\]
Using the fact that $\delta_n$ is increasing we obtain
\[
\delta_{n}(w_n(x)-h(x)- \varepsilon)
\leq
\delta_{n}(M_n) \leq  \delta_{n}(w_n(x)-h(x) +\varepsilon),
\]
and by adding $-\delta_{n}(w_n(x))$, we obtain
\begin{align*}
\delta_{n}(w_n(x)-h(x)- \varepsilon)-\delta_{n}(w_n(x)
&\leq
\delta_{n}(M_n) -\delta_{n}(w_n(x))
\\
&\leq
\delta_{n}(w_n(x)-h(x) +\varepsilon)-\delta_{n}(w_n(x).
\end{align*}
Now, from Assumption~\ref{assump not assymp} it follows that
\[
\delta_{n}(M_n) -\delta_{n}(w_n(x))=-h(x).
\qedhere
\]
\end{proof}

\begin{remark} \label{add cte by u dont change}
All the results of this section were obtained under
the assumption that $u(x,a) \geq 0, \; \forall x, a$.
If we start with a bounded $u$ we can pick a constant $c$ such that $u'=u+c\geq 0$.
We claim that this hypothesis is actually not a restriction neither changes our results.
In the regularity section, all the results depends on
$\hat{u}'(x,y,a)=|(u+c) (x,a)- (u+c)(y,a)|=|u(x,a)- u(y,a)|=\hat{u}(x,y,a)$
and it does not changes under addition of a constant.
In Theorem~\ref{bellman subaction}, we have
$h(x)= \max_{a \in \Psi(x)} u'(x, a)- \bar{u}' +  h(f(x,a)),$ and
$\bar{u}'=\sup_{\mu \in \mathcal{H}} \int_{\Omega} u'(x,a) d\mu(x,a)$,
so $h(x)= \max_{a \in \Psi(x)} u(x, a) -(\bar{u}'-c) +  h(f(x,a))$ and
$\bar{u}'=\sup_{\mu \in \mathcal{H}} \int_{\Omega} u(x,a) +c d\mu(x,a)$
that is
\[
\bar{u}' -c
=
\sup_{\mu \in \mathcal{H}} \int_{\Omega} u(x,a) d\mu(x,a)
=
\bar{u},
\]
thus the equation holds for $u$,
$h(x)= \max_{a \in \Psi(x)} u(x, a) -\bar{u} +  h(f(x,a))$
with the same solution $h$. Analogously, in
Theorem~\ref{ruelle theorem discounted}, if we replace an initial $u$ that
can be negative by $u'=u+c$, we obtain
\begin{align*}
e^{k} e^{h(x)}
&=
\int_{a \in \Psi(x)} e^{u'(x, a)} e^{h(f(x,a))} d\nu_{x}(a)
=
\int_{a \in \Psi(x)} e^{u(x, a) +c } e^{h(f(x,a))} d\nu_{x}(a)
\end{align*}
or equivalently
\[
e^{k-c} e^{h(x)}=\int_{a \in \Psi(x)} e^{u(x, a)} e^{h(f(x,a))} d\nu_{x}(a)
\]
which means that Theorem~\ref{ruelle theorem discounted} holds,
with the same eigenfunction $e^{h(x)}$ and a new eigenvalue $e^{k-c}$.
\end{remark}

\section{Applications to IFS and Related Problems}

\subsection{Subshifts of finite type}
\label{subshift case}
Let $A=\{0,1,\ldots, m-1\}$ be an alphabet and
$C=(c_{ij})_{m\times m}$ an adjacency matrix with entries in $\{0,1\}$.
Let $X=\Sigma_{C} \subseteq A^{\mathbb{N}}$
be the set of all infinite admissible sequences.
To get information about thermodynamic formalism
in the setting of sequence decision-making processes,
for each $x \in \Sigma_{C}$, we put
$\Psi(x)=\{ i \in A \; | \; c_{i,x_{0}}=1\}$
and we recover the dynamics by considering
the maps $f(x,a)=(a, x_0, x_1, \ldots)$,
for each $a \in \Psi(x)$.
Given a H\"older potential $g:X \to \mathbb{R}$,
we define $u(x,a)= g(f(x,a))$.
Considering a variable discount $\delta$ we obtain a
sequential decision-making process $S=\{X, A, \Psi, f, u, \delta\}$.

\subsection{Dynamics of expanding endomorphisms}
\label{Endomorphism case}
Let $(X,d)$ be a complete metric space and $T: X \to X$ a continuous
expanding endomorphism. Suppose that for each point $x \in X$
there is finite set of injective domains $J_{0},J_{1}, \ldots, J_{n-1}$
for $T$.
Take $A:=\{0, 1, \ldots,n-1\}$ and for each $x \in X$ define
$\Psi(x):=\{a \in A \, | \, x \in T^{-1}|_{J_{a}}(X)\}$.
The function $f$ is defined as follows
$f(x,a)=T^{-1}|_{J_{a}}(x)$, for each $a \in \Psi(x)$.
Given a H\"older potential $\varphi: X \to \mathbb{R}$
we define $u(x,a):=\varphi(f(x,a))$ and $\delta(t)=\beta t$.
Then $S=\{X, A, \Psi, f, u, \delta\}$ is a sequential decision-making
process associated with the thermodynamical formalism
for the endomorphism $T$ with a potential $\varphi$.
	
Following the classical approach in thermodynamical formalism as in
\cite{MR1841880} and \cite{MR3377291} we have that
the Bellman and the discounted transfer operators
are given by
\[
B(v)(x):=  \sup_{T(y)=x} \varphi(y)+\beta v(y)
\quad\text{and}\quad
P(v)(x):=\ln \sum_{T(y)=x} e^{\varphi(y)+\beta v(y)}.
\]

In \cite{MR1841880} and \cite{MR3377291} the author shows that
the discounted limit of the first one
provides a calibrated subaction equation:
\[
v_{\infty}(x)
:=
\sup_{T(y)=x} \varphi(y) - m_{\infty} + v_{\infty}(y).
\]
It is well known (see \cite{MR1841880} or \cite{MR3701349})
that a measure $\mu_{max}$ satisfying
\[
m_{\infty}
=
\sup_{T^{*}\mu=\mu} \int_{X} g\, d\mu
=
\int_{X} \varphi\,  d\mu_{max},
\]
is supported in
$\{y \in X \, | \, v_{\infty}(T(y))= \varphi(y)-m_{\infty} + v_{\infty}(y) \}$.
It is also well known
(see \cite{MR1841880} or \cite{MR3377291})
that the discounted limit of the second one gives
a positive eigenfunction $e^{v_{\infty}(x)}$ and an
maximal eigenvalue $e^{k}$ (which is the spectral radius) of the Ruelle operator,
that is,
\[
e^{k} e^{v_{\infty}(x)}=\sum_{T(y)=x} e^{\varphi(y)} e^{v_{\infty}(y)}.
\]

\subsection{IFS with Weights and Thermodynamic Formalism}
\label{IFS case}
Let $(X, d)$ be a complete metric space,
and $(A, d_{A})$ an arbitrary metric space,
indexing a family of continuous functions
$\phi_{a}: X \to X$. Consider the IFS $(X, (\phi_{a})_{a\in A})$.
If in addition, a family of probability measures
$p_{a}: X \to [0,1]$, indexed in $A$, is given one can construct an
ordered triple $(X, (\phi_{a})_{a\in A}, (p_{a})_{a\in A} )$ which is called
an iterated function system with place dependent probabilities (IFSpdp).
To view such IFSpdp as sequential decision-making process
associate to this IFS, we take $\Psi(x)=A, \; \forall x \in X$,
consider the immediate return $u(x,a)=\ln p_{a}(x)$,
which is bounded from above if each $p_{a}$ is so, and is bounded from below
if each $p_{a}(x)>\alpha >0$.
If we consider the  dynamics $f(x,a)=\phi_{a}(x)$ and a discount function $\delta$
then $S=\{X, A, \Psi, f, u, \delta\}$ is a sequential decision-making process
associated with the thermodynamical formalism
of the IFSpdp $(X, \phi_{a}, p_{a})$.

Assuming the hypothesis of Lemmas
\ref{boundness bellman} and \ref{equicontinuous bellman},
Assumption~\ref{assump not assymp}, Theorem~\ref{bellman subaction},
and Remark~\ref{equivalence subaction and HJB equation}
we have that the equation
$b(x)= \max_{a \in \Psi(x)} \ln p_{a}(x)- \bar{u} +  b(f(x,a))$ can be rewritten as
\[
\bar{u}
=
\max_{a \in \Psi(x)} \ln p_{a}(x) +  b(f(x,a)) -b(x)
=
\max_{a \in \Psi(x)} d_{x}b(a) + u(x, a),
\]
where the discrete differential is $d_{x}b(a)= b(f(x,a)) -b(x)$ and
\[
\bar{u}
=
\sup_{\mu \in \mathcal{H}}
\int_{\Omega} \ln p_{a}(x)\,  d\mu(x,a)
\]
where the set of holonomic probabilities is given by
\[
\mathcal{H}
:=
\left\{
\mu \in \mathscr{P}(\Omega)
\; | \;
\int_{\Omega}d_{x}g (a)\, d\mu(x,a)=0, \; \forall g \in C(A,\mathbb{R})
\right\}.
\]

From Theorem~\ref{ruelle theorem discounted} there exists a value
$k \in [-\| u \|_{\infty} , \| u \|_{\infty} ]$ and a function $h$
such that  $\rho:=e^{k}$  and $\varphi:= e^{h(x)}$ are respectively a positive
eigenvalue and a positive and continuous eigenfunction for Ruelle's Operator
\[
e^{k} e^{h(x)}=\int_{a \in \Psi(x)} e^{u(x, a)} e^{h(f(x,a))} d\nu_{x}(a),
\]
or equivalently
\[
\int_{a \in \Psi(x)} p_{a}(x)  e^{h(f(x,a))} d\nu_{x}(a)=e^{k} e^{h(x)}.
\]

As a historical remark we shall mention that
the first version of the Ruele-Perron-Frobenius theorem for contractive IFS,
via shift conjugation, was obtained in \cite{MR1669203}.

\begin{theorem}\label{transl endom to IFS}
The IFS case encompasses the expanding endomorphism case, if
$A=\{0, 1, \ldots,n-1\}$, $\Psi(x)=\{j \in A \, | \, x \in T_{j}(X)\} $,
$f(x,a)=T^{-1}|_{J_{a}}$ and $u(x,a):=\varphi(f(x,a))$,
where $\varphi: X \to \mathbb{R}$ is a
H\"older potential, similarly to Example~\ref{Endomorphism case}.
In this case the IFSpdp
$(X, \phi_{a}, p_{a})$ is such that
$\phi_{a}(x)=f(x,a)$ and $p_{a}(x)=\exp(\varphi(f(x,a)))$.
\end{theorem}
\begin{proof}
We notice that $T\circ f(x,a) =x, \; \forall a \in A$.
The fact that $T$ is uniformly expanding implies that $f(\cdot,a)$
is a uniform contraction and,
the fact that $\varphi$ is H\"older implies bounded and domination conditions
of Theorem~\ref{regularity implies domination} are in hold.
Obviously, Bellman's equation is
\[
b(x)
=
\max_{a \in \Psi(x)} \ln p_{a}(x)- \bar{u} +  b(f(x,a))
=
\max_{a \in \Psi(x)} \varphi(f(x,a))- \bar{u} +  b(f(x,a))
\]
since $u(x,a)=\ln p_{a}(x)=\varphi(f(x,a))$.
It remains to show that $\bar{u}= m_{\infty}$.
Indeed, take any $g$ and  $\mu \in \mathcal{H}$ then
\begin{align*}
0
&=
\int_{\Omega}d_{x}(g\circ T) (a)\, d\mu(x,a)
\\
&=
\int_{\Omega}(g\circ T)(f(x,a)) -(g\circ T)(x)\, d\mu(x,a)
\\
&=
\int_{X}\int_{A} g(x) -g(T(x))\, d\mu_{x}(a)\, d\pi^{*}(\mu)(x)
\\
&=
\int_{X}g(x) -g(T(x))\, d\pi^{*}(\mu)(x),
\end{align*}
where $\pi^{*}$ is the push forward with respect to the projection in $X$.
Thus, $\pi^{*}(\mu)$ is a $T$-invariant measure,
that is $\pi^*(\mathcal{H}) \subseteq \{\eta\; | \; T\eta=\eta\}$. Moreover
\begin{align*}
\int_{\Omega}u(x,a)\, d\mu(x,a)
&=
\int_{\Omega}\varphi(f(x,a))\, d\mu(x,a)
\\
&=
\int_{\Omega}\varphi(f(x,a))-\varphi(x) + \varphi(x)\, d\mu(x,a)
\\
&=
\int_{\Omega}d_{x}(\varphi)(a) d\mu(x,a) + \int_{X} \varphi(x)\,
d\pi^{*}(\mu)(x)
\\
&=
0+ \int_{X} \varphi(x)\, d\pi^{*}(\mu)(x),
\end{align*}
thus showing that $\bar{u} \leq m_{\infty}$.
To obtain the equality we will construct a special
holonomic measure $\mu \in \mathcal{H}$ satisfying
\[
\int_{\Omega}u(x,a) d\mu(x,a)=m_{\infty}.
\]
We first observe that from the calibrated subaction equation
\[
v_{\infty}(x)=  \sup_{T(y)=x} \varphi(y) - m_{\infty} + v_{\infty}(y)
\]
which is equivalent to
\[
m_{\infty}=  \sup_{T(y)=x} \varphi(y) +  d_x v_{\infty}(y)
\]
we can obtain, proceeding similarly as in Theorem~\ref{ergodic optim for IFS},
an optimal holonomic measure $\mu \in \mathcal{H}$, such that
\[
\int_{\Omega}u(x,a) d\mu(x,a)=m_{\infty}.
\]
This shows that the calibrated subaction equation is
equivalent to the associated Bellman's equation.

Finally, notice that the equation
\[
\int_{a \in \Psi(x)} p_{a}(x)  e^{h(f(x,a))} d\nu_{x}(a)=e^{k} e^{h(x)},
\]
for the IFS is equivalent to,
\[
\int_{a \in \Psi(x)} e^{\varphi(f(x,a))}  e^{h(f(x,a))} d\nu_{x}(a)=e^{k} e^{h(x)}.
\]
Recalling that $T\circ f(x,a) =x, \; \forall a \in A$, we obtain
\[
\sum_{T(y)=x} e^{\varphi(y)-\ln{n}}  e^{h(y)}=e^{k} e^{h(x)},
\]
where $\nu_{x}(a)=(1/n)\sum_{j}\delta_{j}(a)$,
which is the same operator as considered in the endomorphism
case, up to the constant $(-\ln{n})$.
\end{proof}

\section*{Acknowledgments}
This study was financed in part by the Coordena\c c\~ao de Aperfei\c coamento
de Pessoal de N\'ivel Superior - Brasil (CAPES) - Finance Code 001.
L. Cioletti would like to acknowledge financial support by CNPq through project 310818/2015-0.

\newcommand{\etalchar}[1]{$^{#1}$}

\end{document}